\documentclass{article}
\usepackage{amsmath,amsfonts,amssymb,amsopn,amscd,amsthm,accents}
\usepackage{bbm,wasysym,stmaryrd}
\usepackage{hyperref,color}
\usepackage[normalem]{ulem}
\usepackage{graphicx}
\usepackage[nocompress]{cite}

\newcommand{\Exp}{\mathbbm{E}}
\newcommand{\E}{\mathcal{E}}

\newcommand{\R}{\mathbbm{R}}
\newcommand{\N}{\mathbbm{N}}

\newcommand{\F}{\mathcal{F}}	

\newcommand{\f}{\varphi}

\newcommand{\der}[2]{ \frac{\text{d} #1}{\text{d} #2} }  

\newcommand{\M}{\mathcal{M}}
\newcommand{\C}{\mathcal{C}}

\newcommand{\bmax}{\lVert b \rVert_{\infty}}

\theoremstyle{plain}
\newtheorem{theorem}{Theorem}

\newtheorem{proposition}[theorem]{Proposition}
\newtheorem{lemma}[theorem]{Lemma}

\newtheorem{remark}{\it Remark\/}

\begin{document}

\title{Large deviations for randomly connected neural networks: II. State-dependent interactions}
\date{\today}
\author{Tanguy Cabana \and Jonathan Touboul}
\maketitle

\begin{abstract}
This work continues the analysis of large deviations for randomly connected neural networks models of the brain. The originality of the model relies on the fact that the directed impact of one particle onto another depends on the state of both particles, and (ii) have random Gaussian amplitude with mean and variance scaling as the inverse of the network size. Similarly to the spatially extended case, we show that under sufficient regularity assumptions, the empirical measure satisfies a large-deviation principle with good rate function achieving its minimum at a unique probability measure, implying in particular its convergence in both averaged and quenched cases, as well as a propagation of chaos property (in the averaged case only). The class of model we consider notably includes a stochastic version of Kuramoto model with random connections. 
\end{abstract}

\bigskip

\hrule

\tableofcontents

\medskip

\hrule

\section{Introduction}
We pursue our study of randomly connected neural networks inspired from neurobiology. In the companion paper~\cite{cabana-touboul:15}, we studied spatially extended neural networks with space-dependent delays and random interactions with mean and variance depending on cell locations, and scaling as the inverse of the network size. In that model, as well as in all previous works dealing with similar interaction coefficients scalings, the fact that the interaction between two particles may depend on the state of both particles was neglected. However, it is now well admitted that the interactions between two neurons depend on the state of both the pre-synaptic and post-synaptic cell~\cite{ermentrout-terman:10b,gerstner-kistler:02}. This type of state-dependent interaction is much more general and actually ubiquitous in life science including models collective animal behaviors~\cite{carrillo-fornasier:10} or a natural coupled oscillators such as those described by the canonical Kuramoto model~\cite{kuramoto:75, strogatzSync,strogatz2000kuramoto,lucconthesis}. The present manuscript addresses the dynamics of networks with state-dependent interactions and random coupling amplitudes, in a general setting. In detail, we consider the interaction of $N$ agents described by a real state variable $(X^{i,N}_t)_{i=1\cdots N}\in \R^N$ and satisfying a stochastic differential equation of the type of~\cite[equation (1)]{cabana-touboul:15}:
\begin{equation}\label{eq:Standard equation}
	dX^{i,N}_t=\left(f(r_i,t,X^{i,N}_t) + \sum_{j=1}^{N} J_{ij} b(X^{i,N}_t, X^{j,N}_t) \right)\,dt+\lambda dW^{i}_t,
\end{equation}
where $f$ describes the intrinsic dynamics of the particle, $J_{ij}$ models the random interaction amplitude, $b(x,y)$ is the typical impact of a particle with state $y$ on a particle with state $x$, and each particle is subject to independent Brownian fluctuations $(W^{i}_{t})$, see~\cite{cabana-touboul:15} for details on this equation.
 
Following the general methodology introduced in~\cite{ben-arous-guionnet:95,guionnet:97,ben-arous-dembo-guionnet:01,ben-arous-guionnet:97} also used in the companion paper~\cite{cabana-touboul:15}, we will show using large-deviations techniques that the empirical measure of system~\eqref{eq:Standard equation}, averaged over the disorder parameters, satisfies a Large Deviations Principle (LDP), with an explicit good rate function that has a unique minimum implying convergence of the network equations towards a non-Markovian complex mean-field equation. Taking into account general interactions introduces a number of specific difficulties compared to previous works. In particular, the dependence in the state of the particle induces complex interdependences between processes that prevents from isolating exponential martingales terms as done when $b(x,y)=S(y)$. We will handle this issue using specific estimates that will lead us to restrict the time horizon.

The paper is organized as follows. We start by introducing the mathematical setting and main results in Section~\ref{sec:MathSetting}. The proofs are found in the following sections. Section~\ref{sec:LDP} establishes a partial LDP for the averaged empirical measure, which relies on the identification of the good rate function as well as on exponential tightness and upper-bounds on closed sets for the sequence of empirical measures${}^{1}$\footnote{${}^{1}$ In the companion paper~\cite{cabana-touboul:15} we only proved tightness and upper-bounds for compact sets to avoid any constraint on time horizon.}. In Section~\ref{sec:LimitIdentification}, we demonstrate that the good rate function admits a unique minimum $Q$ and prove the averaged and quenched convergence of the empirical measure towards $Q$ using the methodology introduced in~\cite{cabana-touboul:15}. Eventually, we discuss a few perspectives as well as some open research directions in the conclusion.

\section{Mathematical setting and statement of the results}\label{sec:MathSetting}

The general mathematical setting was introduced in \cite[Section 2]{cabana-touboul:15}. The aim of this second part is to cope with complex interactions of the form $b(x,y)$ which depend on the state of both particles. 

In order to expedite the analysis, we neglect spatial aspects already addressed in the first part, and particularly (i) consider the synaptic coefficients identically distributed with law $\mathcal{N}\big(\frac{\bar{J}}{N}, \frac{\sigma^2}{N} \big)$, (ii) diffusion coefficients independent of space and (iii) no interaction delay. Formally, this amounts assuming in the general framework of \cite[Section 2]{cabana-touboul:15} that $J(r,r') \equiv \bar{J} \in \R$, $\sigma(r,r') \equiv \sigma \in \R^*_+$ and $\lambda(r) \equiv \lambda>0$ and $\tau(r,r') \equiv 0$. Therefore, the initial conditions are real variables ($C_{\tau}=\R$) and the trajectories belong to $\C = \C \big( [0,T], \R \big)$.

Our results will hold under the condition that the horizon of time $T$ is such that 
\begin{equation}\label{TimeCondition}
\frac{2\sigma^2 \bmax^2 T}{\lambda^2} < 1.
\end{equation}
Compared to the results of the companion paper~\cite{cabana-touboul:15}, this condition also proving stronger results on $Q^N\big(\hat{\mu}_N \in \cdot \big)$: an  \emph{exponential} tightness (Theorem~\ref{thm:Convergence}) and an upper-bound for \emph{closed sets} (and not restricted to compact sets, Theorem~\ref{thm:WeakLDP}).

These may be summarized as follows:
\begin{theorem}\label{thm:Convergence}
For $T$ small enough for inequality~\eqref{TimeCondition} to hold, there exists a unique double-layer probability distribution $Q \in \M_1^+(\C \times D)$ such that:
\[
Q^{N}(\hat{\mu}_{N}\in \cdot) \overset{\mathcal{L}}{\to} \delta_Q(\cdot) \in \M_1^+\big(\M_1^+(\C \times D)\big),
\] 
exponentially fast.
\end{theorem}

The existence of $Q$ and the exponential convergence results follow from three points: (i) the exponential tightness of the sequence $Q^N\big(\hat{\mu}_N \in \cdot \big)$, (ii) a partial LDP for the empirical measure relying on an upper-bound for closed sets, and (iii) a characterization of the set of minima of the good rate function. 
\begin{theorem}[Partial Large Deviation Principle]\label{thm:WeakLDP} $\ $\\
For $T$ small enough for inequality~\eqref{TimeCondition} to hold,
\begin{enumerate}
\item for any real number $M \in \R$, there exists a compact subset $K_M$ such that for any integer $N$,
\[
\frac{1}{N} \log Q^{N}(\hat{\mu}_N \notin K_M)\leq -M.
\]
\item there exists a good rate function $H: \M_1^+ (\C \times D)$ such that for any closed subset $F$ of $\M_1^+ (\C \times D)$:
\[
\limsup_{N\to\infty} \frac 1 N \log Q^N(\hat{\mu}_{N}\in F )\leq -\inf_{F} H.
\]
\end{enumerate}
\end{theorem}
This theorem is proved in section~\ref{sec:LDP}.

\begin{theorem}[Minima of the rate function]\label{thm:Limit}
  The good rate function $H$ achieves its minimal value at a unique probability measure $Q \in \M_1^+ (\C \times D)$ satisfying:
    \[Q \simeq P, \qquad \der{Q}{P}(x,r)=\E \left[\exp\bigg\{\frac 1 {\lambda}\int_0^T G_t^{Q}(x) dW_t(x,r) - \frac 1 {2\lambda^2} \int_0^T (G_t^{Q}(x))^2 dt\bigg\} \right]\]
    where $(W_t(.,r))_{t\in [0,T]}$ is a $P_r$-Brownian motion, and $G^{Q}(x)$ is a $(\tilde{\Omega},\tilde{\F},\mathcal{P})$, a Gaussian process with mean:
    \[
    \E [G^{Q}_t(x)] = \int_{\C \times D} \bar{J}  b(x_t,y_t)dQ(y,r') 
    \]
    and covariance:
    \[
    \E [G^{Q}_t(x)G^{Q}_s(x)] = \int_{\C \times D} \sigma^2 b(x_t, y_t)b(x_s, y_s)dQ(y,r').
    \]
\end{theorem}
This theorem will be demonstrated in section~\ref{sec:LimitIdentification}. Combining both results, the general result of Sznitman~\cite[Lemma 3.1]{sznitman:84} implies that:
\begin{theorem}[Propagation of chaos]\label{thm:PropagationOfChaos}
	For $T$ small enough for inequality~\eqref{TimeCondition} to hold, $Q^N$ is $Q$-chaotic in the sense that for any  $m \in \N^*$, any collection of bounded continuous functions $\f_1,\ldots,\f_m: \C \times D \to \R$ and any set of nonzero distinct integers $k_1,\ldots, k_m$, we have:
	\[
	\lim_{N\to\infty}\int_{\big(\C \times D \big)^N} \prod_{j=1}^m \f_j(x^{k_j},r_{k_j})dQ^N(\mathbf{x},\mathbf{r}) = \prod_{j=1}^m \int_{\C \times D} \f_j(x,r)dQ(x,r). 
	\]
\end{theorem}

Our results partially extends to the quenched case as stated in the following theorem:
\begin{theorem}[Quenched results]\label{thm:Quenched}
For $T$ small enough for inequality~\eqref{TimeCondition} to hold, we have the following quenched upper-bound:
\[
\mathcal{P}-a.s., \; \; \forall \text{ closed } F \subset \M_1^+(\C \times D), \; \; \limsup_{N\to\infty} \frac 1 N \log Q^N_{\mathbf{r}}(J)(\hat{\mu}_{N}\in F )\leq -\inf_{F} H,
\]
where $H$ is the good rate function introduced in theorem~\ref{thm:WeakLDP}. 
In particular, for almost every realization of $\mathbf{r}$ and $J$, $Q_{\mathbf{r}}^N(J)(\hat{\mu}_N \in \cdot)$ is exponentially tight and converges in law toward $\delta_Q$ exponentially fast. Eventually, this implies $\mathcal{P}$-almost sure convergence of the empirical measure to $Q$.
\end{theorem}

\section{Large deviation principle}\label{sec:LDP}
This section is devoted to proving the existence of a partial large deviations principle for the averaged empirical measure. We start by constructing the appropriate good rate function before obtaining an upper-bound and an exponential tightness result. Many points of the proof proceed as in the companion paper \cite{cabana-touboul:15} or as in precedent works \cite{ben-arous-guionnet:95,guionnet:97,cabana-touboul:12}. To avoid reproducing fastidious demonstrations we will often refer to these contributions and focus our attention on the new difficulties arising in this state-dependent interactions setting.

\subsection{Construction of the good rate function}\label{sec:GoodRate}

For $\mu \in {\M_1^+(\C \times D)}$, we define the two following functions respectively on $[0,T]^2 \times \C$ and $[0,T] \times \C$:
\begin{equation*}
\begin{cases}
	K_{\mu}(s,t,x)&:=\displaystyle{\frac{\sigma^2}{\lambda^2} \int_{\C \times D}  b(x_t,y_t)b(x_s,y_s) d\mu(y,r')}\\
	m_{\mu}(t,x) &:= \displaystyle{\frac{\bar{J}}{\lambda} \int_{\C \times D}  b(x_t,y_t)d\mu(y,r')}.\\
\end{cases}
\end{equation*}
Both functions are well defined as $(y,r) \to b(x_t,y_t)b(x_s,y_s)$ and $(y,r) \to b(x_t,y_t)$ are continuous for the uniform norm on $\C \times D$, and $\mu$ is a Borel measure. They are bounded: $|K_{\mu}(s,t,x)| \leq \frac{\sigma^2 \bmax^2}{\lambda^2}$ and $|m_{\mu}(t,x)| \leq \frac{\bar{J} \bmax}{\lambda}$. They are also continuous by the dominated convergence theorem. 

Since $K_{\mu}$ has a covariance structure, we can define a probability space $(\hat{\Omega},\hat{\F},\gamma)$ and a family of stochastic processes $\big(G^{\mu}(x)\big)_{x \in \C, \mu \in \M_1^+(\C \times D)}$ such that $G^{\mu}(x)$ is a centered Gaussian process with covariance $K_{\mu}(.,.,x)$ under measure $\gamma$. We denote $\E_{\gamma}$ the expectation under $\gamma$.

\begin{remark}\label{rem:structureG} $\ $ 
For the sake of measurability under Borel measures of $\M_1^+\big( \C \times D \big)$, it is convenient to choose a family $\big(G^{\mu}_t(x)\big)_{\mu,x}$ regular in $x$, which can be done by constructing the process explicitly as in \cite[Remark 2.14]{guionnet:97}. In detail, for $\mu \in \M_1^+\big(\C \times D\big)$, and $(e^{\mu}_i)_{i \in \N^*}$ an orthonormal basis of $L^2_{\mu}\big( \C \times D \big)$. Let also for any $x \in \C, t \in [0,T]$, $\rho_{t,x} \in L^2_{\mu}\big( \C \times D \big)$ such that $\rho_{t,x}(y,r):=b(x_t,y_t)$. Then the process
\[
G^{\mu}_t(x) : = \sum_{i \in \N} J_i \langle \rho_{t,x}, e^{\mu}_i \rangle_{L^2_{\mu}( \C \times D )} = \sum_{i \in \N} J_i \int_{\C\times D} b(x_t,y_t) e^{\mu}_i(y,r) d\mu(y,r),
\]
where $\big(J_i\big)_{i \in \N^*}$ are independent centered Gaussian variables of the probability space $(\hat{\Omega},\hat{\F},\gamma)$ and with variance $\sigma^2$ provides a regular representation of the process $(G^{\mu}_t(x))_{t\geq 0}$.
\end{remark}

We recall that for any Gaussian process $(G_t)_{t \in [0,T]}$ of $\big(\hat{\Omega}, \hat{\mathcal{F}}, \gamma \big)$, and any $t \in [0,T]$
\begin{equation}\label{def:Lambda}
\Lambda_t(G):=\frac{ \exp\Big\{ - \frac{1}{2} \int_0^t G_s^2 ds \Big\}}{\E_{\gamma} \Big[\exp\Big\{ - \frac{1}{2} \int_0^t G_s^2 du\Big\} \Big]},
\end{equation}
and also define for any $\nu \in \M_1^+\big(\C \times D\big)$, $(x,r) \in \C \times D$ and $t \in [0,T]$
\begin{equation}\label{LnuVnuDef}
L^{\nu}_{t}(x,r) := \int_0^t G^{\nu}_s(x)\Big( dW_s(x,r) - m_{\nu}(s,x) ds\Big), \quad V^{\nu}_t(x,r):= W_t(x,r) - \int_0^t m_{\nu}(s,x) ds.
\end{equation}
Moreover, we introduce
\[
d\gamma_{\widetilde{K}_{\nu,x}^{t}}(\omega) :=  \Lambda_t(G^{\nu}(\omega,x)) d\gamma(\omega), \quad \forall \omega \in \hat{\Omega}.
\]
for any $t \in [0,T], x \in \C, \nu \in \M_1^+(\C \times D)$. It is proven in \cite{neveu:70} that $\gamma_{\widetilde{K}_{\nu,x}^{t}}$ is a probability measure, under which $G^{\nu}(x)$ is a centered Gaussian process with covariance:
\[
\widetilde{K}_{\nu,x}^{t}(s,u):=\E_{\gamma} \bigg[ G^{\nu}_u(x)G^{\nu}_s(x) \Lambda_t\big(G^{\nu}(x)\big) \bigg].
\]

Moreover, let for any fixed $N \in \N^*$, $\forall (\mathbf{x},\mathbf{r}) \in (\C \times D)^N$
\[
X^N_i(\mathbf{x},\mathbf{r}):= \int_0^T G^{i,N}_t(\mathbf{x}) dW_t(x^i,r_i)-\frac{1}{2} \int_0^T {G^{i,N}_t(\mathbf{x})}^2 dt
\]
where $G^{i,N}_t(\mathbf{x}):=\frac{1}{\lambda} \sum_{j=1}^N J_{ij} b(x^i_t,x^j_t)$. As proven in \cite{cabana-touboul:15}, we have the following good properties:

\begin{proposition}\label{LambdaProperties}
Exists a constant $C_T>0$, such that for any $\nu \in \M_1^+\big(\C \times D \big)$, $x \in \C$, $t \in [0,T]$,

\begin{equation}\label{ineq:BoundOnLambda}
	\underset{0\leq s,u \leq t}{\sup} \widetilde{K}_{\nu,x}^{t}(s,u) \leq C_T,\quad 
	\Lambda_t\big(G^{\nu}(x) \big) \leq C_T,
	\end{equation}

	\begin{align}\label{eq:KtildeExpectation}
	\E_{\gamma} \bigg[ \exp\bigg\{-\frac{1}{2} \int_0^T {G^{\nu}_t(x)}^2 dt\bigg\} \bigg] = \exp\Big\{ -\frac{1}{2} \int_0^T \widetilde{K}_{\nu,x}^t(t,t) dt \Big\}.
	\end{align}
Moreover, if $(G_t)_{0\leq t \leq T}$ and $(G'_t)_{0\leq t \leq T}$ are two centered Gaussian processes of $\big(\hat{\Omega},\hat{\mathcal{F}},\gamma\big)$ with uniformly bounded covariance, then exists $\tilde{C}_T>0$ such that for all $t \in [0,T]$,
\begin{align}\label{ineq:DiffLambda}
\big| \Lambda_t(G)-\Lambda_t(G') \big| \leq \tilde{C}_T  \bigg\{ \int_0^t \E_{\gamma} \Big[ \big(G_s - {G'_s}\big)^2 \Big]^{\frac{1}{2}} ds + \int_0^t \big|G_s^2 - {G'_s}^2\big| ds \bigg\}.
\end{align}
\end{proposition}

\begin{lemma}\label{lemma2}
\begin{equation*}
\frac{dQ^N}{dP^{\otimes N}}(\mathbf{x},\mathbf{r}) = \exp{\Big\{ N \bar{\Gamma}(\hat{\mu}_N) \Big\} }.
\end{equation*}
where,
\begin{equation*}
\bar{\Gamma}(\hat{\mu}_N) :=  \frac{1}{N} \sum_{i=1}^N \log \E_{\gamma} \bigg[ \exp\bigg\{ \int_0^T \big(G^{\hat{\mu}_N}_t(x^i)+m_{\hat{\mu}_N}(t,x^i)\big)dW_t(x^i,r_i) - \frac{1}{2} \int_0^T \big(G^{\hat{\mu}_N}_t(x^i)+m_{\hat{\mu}_N}(t,x^i)\big)^2dt\bigg\}  \bigg].
\end{equation*} 
\end{lemma}

As in \cite{cabana-touboul:15}, this last lemma suggests that a version of Varadhan's lemma ans a LDP might hold. The following lemma properly defines the associated Varadhan's functional:

\begin{proposition}\label{prop:GammaBehaviour}
Let
\[
X^{\mu}(x,r):= \int_0^T \big(G^{\mu}_t(x)+m_{\mu}(t,x)\big) dW_t(x,r)-\frac{1}{2} \int_0^T \big(G^{\mu}_t(x)+m_{\mu}(t,x)\big)^2 dt.
\]
The map
\begin{align}\label{eq:Gamma}
\Gamma := \mu \in \M_1^+\big(\C \times D \big) \to
\left\{
\begin{array}{cl}
{\int_{\C \times D} \log \E_{\gamma} \Big[ \exp\big\{ X^{\mu}(x,r) \big\}  \Big] d\mu(x,r)} & \text{if } I(\mu|P) < \infty,\\
+ \infty & \text{otherwise }.
\end{array}
\right.
\end{align}
is well defined in $\R \cup \{ + \infty\}$, and satisfies
\begin{enumerate}
\item $\Gamma \leq I(\cdot|P)$,
\item If $\frac{2 \sigma^2 \bmax^2 T}{\lambda^2} < 1$, $\exists \iota \in ]0, 1[, e \geq 0$, $|\Gamma(\mu)| \leq \iota I(\mu|P) + e$.
\end{enumerate}
\end{proposition}

The proofs slightly differs from that of \cite[Proposition 9]{cabana-touboul:15} as the dependence of the Gaussian process $G^{\mu}$ in $x$ prevents from extracting it from integral over $P_r$. We thus reproduce the important lines of the proof, and rely on H\"older's inequality to cope with this new problem.

\begin{proof}
We suppose that $I(\mu|P)<+\infty$ and $\mu \ll P$ as the results is otherwise trivial. As $W(\cdot,r)$ is a $P_r$-Brownian motion, Girsanov's theorem ensures that the stochastic integral $\int_0^T \big(G^{\mu}_t(x)+m_{\mu}(t,x) \big)dW_t(x,r)$ is well defined $\gamma$-almost surely under $\mu$.\\
\noindent {\bf (1):}\\
Following the exact same proof as in \cite[Proposition 9]{cabana-touboul:15}, we obtain for any $\alpha \geq 1$:
\begin{multline}\label{ineq:GammaBehavior1}
\alpha \int_{\C \times D} \log\Big( \E_{\gamma} \big[ \exp\{X^{\mu}(x,r)\} \big] \vee M^{-1} \Big) d\mu(x,r) \leq 
\\ I(\mu|P) + \log\bigg\{ M^{-\alpha} + \E_{\gamma} \bigg[ \int_{D} \int_{\C} \exp\Big\{\alpha X^{\mu}(x,r) \Big\} dP_r(x) d\pi(r) \bigg] \bigg\},
\end{multline}
\begin{multline}\label{ineq:GammaBehavior2}
\alpha \int_{\C \times D} \Big( \log \E_{\gamma} \big[ \exp\big\{X^{\mu}(x,r)\big\} \big] \Big)^- d\mu(x,r) \leq \\
I(\mu|P) + \alpha C_T + \log\bigg\{ \E_{\gamma} \bigg[ \int_{D} \int_{\C} \exp\Big\{\alpha X^{\mu}(x,r) \Big\} dP_r(x) d\pi(r) \bigg] \bigg\},
\end{multline}
with the right-hand side of these two inequalities being possibly infinite. 	Moreover, $W(.,r)$ being a $P_r$-Brownian motion, the martingale property yields for $\alpha=1$
\begin{equation*}
\E_{\gamma} \bigg[ \int_{D} \int_{\C} \exp\Big\{\alpha X^{\mu}(x,r) \Big\} dP_r(x) d\pi(r) \bigg] = 1,
\end{equation*}
so that we obtain the first point by that sending $M \to + \infty$.

\noindent{\bf (2):}\\
Let $\alpha>1$, and $(\alpha,\frac{\alpha}{\alpha-1})$ be conjugate exponents. Then, making use of a martingale property:
\begin{align*}
\int_{\C \times D} \! \! \! \E_{\gamma} \Big[ & \exp\big\{\alpha X^{\mu}(x,r) \big\}  \Big] dP(x,r) \leq 
 \bigg\{ \int_{\C \times D} \! \! \! \! \E_{\gamma} \bigg[ \exp\Big\{\frac{\alpha^2(\alpha+1)}{2} \int_0^T \big(G^{\mu}_t(x) + m_{\mu}(t,x) \big)^2 dt \Big\}  \bigg] dP(x,r) \bigg\}^{\frac{\alpha-1}{\alpha}}.
\end{align*}
Under the short-time hypothesis $\frac{2 \sigma^2 \bmax^2 T}{\lambda^2} < 1$, we can proceed as in \cite[Proposition 9]{cabana-touboul:15} to prove finiteness of the right-hand side for $\alpha-1$ small enough, as we are able to rely on the following identity, valid for $\zeta\sim\mathcal{N}(\alpha,\beta)$ with $\beta<1$:
\begin{equation}\label{eq:GaussianExponentialQuadraticMoment}
\Exp \Big[ \exp\Big\{ \frac{1}{2} \zeta^2 \Big\}\Big]=\frac{1}{\sqrt{1-\beta}}\exp\Big\{\frac{\alpha^2}{2(1-\beta)} \Big\}=\exp\Big\{\frac{1}{2}\Big(\frac{\alpha^2}{1-\beta} - \log(1-\beta) \Big) \Big\}.
\end{equation}
Hence, by Jensen and Fubini's inequalities, we obtain that there exists a constant $C_T$ uniform in $x \in \C$ such that:
\[
\E_{\gamma} \bigg[ \exp\Big\{\frac{\alpha^2(\alpha+1)T}{2} \int_0^T \big(G^{\mu}_t(x) + m_{\mu}(t,x) \big)^2 \frac{dt}{T} \Big\}  \bigg] \leq  \exp\{C_T\},
\]
implying:
\begin{equation}\label{ineq:GammaBehavior3}
 \int_{\C \times D} \E_{\gamma} \Big[ \exp\big\{\alpha X^{\mu}(x,r) \big\}  \Big] dP(x,r)  \leq \exp\big\{(\alpha-1)C_T \big\}.
\end{equation}

Inequalities \eqref{ineq:GammaBehavior1}, \eqref{ineq:GammaBehavior2}, and \eqref{ineq:GammaBehavior3} ensure that, under the condition $\frac{2 \sigma^2 \bmax^2 T}{\lambda^2} < 1$ and for $\alpha>1$
\begin{equation*}
|\Gamma(\mu)| \leq \iota I(\mu|P)+ e,
\end{equation*}
with $\iota:=\frac{1}{\alpha}$, and $e:=(2\alpha-1) C_T$.
\end{proof}

Define
\begin{align*}
H (\mu)
:=
\left\{
\begin{array}{ll}
I(\mu|P) - \Gamma(\mu) & \text{if } I(\mu|P) < \infty,\\
\infty & \text{otherwise },
\end{array}
\right.
\end{align*}
for any $\nu \in \M_1^+\big(\C \times D \big)$:
\begin{align*}
\Gamma_{\nu} := \mu \in \M_1^+\big(\C \times D \big) \to
\left\{
\begin{array}{cl}
{\int_{\C \times D} \log \E_{\gamma} \Big[ \exp\big\{ X^{\nu}(x,r) \big\}  \Big] d\mu(x,r)} & \text{if } I(\mu|P) < \infty,\\
+ \infty & \text{otherwise },
\end{array}
\right.
\end{align*}
\begin{equation*}
H_{\nu} : \mu \to\\
\left\{
\begin{array}{ll}
I(\mu|P) - \Gamma_{\nu}(\mu)  & \text{if } I(\mu|P) < + \infty,\\
+ \infty & \text{otherwise, }
\end{array}
\right.
\end{equation*}
as well as the following probability measure on $\C \times D$
\begin{equation}\label{def:Qnu}
dQ_{\nu}(x,r):=\exp\big\{ \bar{\Gamma}_{\nu}(\delta_{(x,r)}) \big\} dP(x,r) := \E_{\gamma} \Big[ \exp\Big\{ X^{\nu}(x,r) \Big\}\Big]dP(x,r).
\end{equation}

We can show as in~\cite[Theorem 11]{cabana-touboul:15} this relatively intuitive result:

\begin{theorem}\label{thm:Qnu}
$Q_{\nu}$ is a well defined probability measure on $\M_1^+(\C \times D)$, and the two maps $H_{\nu}$ and $I(.|Q_{\nu})$ are equal on $\M_1^+(\C \times D)$. In particular $H_{\nu}$ is a good rate function reaching its unique minimum at $Q_{\nu}$.
\end{theorem}

We introduce the Vaserstein distance on $\M_1^+(\C \times D)$, compatible with the weak topology:
\begin{equation*}
d_T^V(\mu,\nu):=\inf_{\xi}\bigg\{ \int_{(\C \times D)^{2}} \left [\sup_{0 \leq t \leq T} \lVert x-y\rVert_{\infty,T}^2+\Vert r-r' \Vert_{\R^d}^2\right] d\xi\big((x,r),(y,r')\big) \bigg\}^{\frac{1}{2}}
\end{equation*}
where the infimum is taken on the laws $\xi \in \M_1^+\big((\C \times D)^2 \big)$ with marginals $\mu$ and $\nu$. Moreover we will denote for any $t \in [0,T]$, and any $(x,r),(y,r') \in \C \times D$,
\[
d_t\big((x,r),(y,r')\big) := \Big(\Vert x-y \Vert_{\infty,t}^2+\Vert r-r' \Vert_{\R^d}^2\Big)^{\frac{1}{2}}
\]
where we recall that $\Vert x-y \Vert_{\infty,t}:=\sup_{0 \leq s \leq t} |x_s-y_s|^2$, and also 
\begin{equation*}
d_t^V(\mu,\nu):=\inf_{\xi}\bigg\{ \int_{(\C \times D)^{2}} d_t\big((x,r),(y,r')\big)^2 d\xi\big((x,r),(y,r')\big) \bigg\}^{\frac{1}{2}}.
\end{equation*}
The metric $d_T^V$ will control the regularity of the mean and variance structure of the Gaussian interactions and, in the long run (see Theorem~\ref{lemma1}), of the error between $H$ and its approximation $H_{\nu}$:

\begin{proposition}\label{prop:MeanAndVarRegularity}
Exists $C_T>0$ such that for any $\mu, \nu \in \M_1^+(\C \times D)$, $x \in C$, $t \in [0,T]$ and $u,s \in [0,t]$:
\begin{multline}\label{ineq:MajDiff}
   \big|m_{\mu}(t,x)-m_{\nu}(t,x)\big| +\big|K_{\mu}(t,s,x)-K_{\nu}(t,s,x)\big|
   + \big|\widetilde{K}^{t}_{\mu,x}(s,u)-\widetilde{K}^{t}_{\nu,x}(s,u)\big| \leq C_T d_T^V(\mu,\nu).
\end{multline}
\end{proposition}

\begin{proof}
First, observe that for any $\xi \in \M_1^+\big((\C \times D)^2 \big)$ with marginals $\mu$ and $\nu$:
    \begin{align*}
	\big|m_{\mu}(t,x)-m_{\nu}(t,x)\big| & = \bigg|\frac{\bar{J}}{\lambda} \int_{\C \times D} b(x_t,y_t) d(\mu-\nu)(y,r')\bigg| \leq \frac{\bar{J}}{\lambda}  \int_{(\C \times D)^2} \Big|b(x_t,y_t) - b(x_t,z_t)\Big| d\xi\big((y,r'),(z,\tilde{r}')\big) \\
	& \overset{\mathrm{C.S}}{\leq} \frac{\bar{J} K_b}{\lambda} \bigg\{ \int_{(\C \times D)^2} \Vert y - z \Vert_{\infty,t}^2 d\xi\big((y,r'),(z,\tilde{r}')\big) \; \bigg\}^{\frac{1}{2}},
	\end{align*}
  so that $\big|m_{\mu}(t,x)-m_{\nu}(t,x)\big| \leq C_T d_T^V(\mu,\nu)$.

Fix now $\xi \in \M_1^+\big((\C \times D)^2 \big)$, with marginals $\mu$ and $\nu$. Letting $\big(G, G'\big)$ be a $\gamma$-bidimensional centered Gaussian processes with covariance:
\begin{equation}
K_{\xi}(s,t,x):= \frac{\sigma^2}{\lambda^2} \int_{(\C \times D)^2} 
\left(
\begin{array}{ccc}
b(x_s,y_s)b(x_t,y_t)  & b(x_s,y_s)b(x_t,z_t) \\
b(x_s,z_s)b(x_t,y_t) & b(x_s,z_s)b(x_t,z_t) \\
\end{array}
\right) 
d\xi \big(y,r'),(z,\tilde{r}')\big) \label{kxi},
\end{equation}
we obtain easily obtain (see \cite[Proof of Proposition 5]{cabana-touboul:15}) the inequalities
\begin{align*}
\big|K_{\mu}(t,s,x) - K_{\nu}(t,s,x) \big| \leq  C_T \bigg\{\E_{\gamma} \Big[ \big(G_t-G_t'\big)^2 \Big]^{\frac{1}{2}} +\E_{\gamma} \Big[ \big(G_s-G_s'\big)^2 \Big]^{\frac{1}{2}}\bigg\},
\end{align*}
\begin{align*}
\big|\widetilde{K}^t_{\mu,x}(s,u) - \widetilde{K}^t_{\nu,x}(s,u) \big| & \overset{\eqref{ineq:BoundOnLambda},\eqref{ineq:DiffLambda}}{\leq} C_T \bigg\{ \bigg(\int_0^t \E_{\gamma} \Big[ \big(G_v-G'_v\big)^2 \Big] dv \bigg)^{\frac{1}{2}} + \E_{\gamma} \Big[ \big(G_s-G_s'\big)^2 \Big]^{\frac{1}{2}} +\E_{\gamma} \Big[ \big(G_u-G_u'\big)^2 \Big]^{\frac{1}{2}} \bigg\}.
\end{align*}

Remarking that
\begin{align*}
\E_{\gamma} \Big[ \big(G_t-G_t'\big)^2 \Big] & = \frac{\sigma^2}{\lambda^2}  \int_{(\C \times D)^2 } \Big(b(x_t,y_t)-b(x_t,z_t)\Big)^2 \, d\xi\big((y,r'),(z,\tilde{r}')\big)\\
& \leq \frac{\sigma^2 K_b^2}{\lambda^2} \int_{(\C \times D)^2 }  d_T\big((y,r'),(z,\tilde{r}') \big)^2 d\xi\big((y,r'),(z,\tilde{r}')\big).
\end{align*}
and taking the infimum over $\xi$ brings the result.
\end{proof}

The following theorem control the error between $H$ and $H_{\nu}$ and ensures that the former is a good rate function under the time condition \eqref{TimeCondition}:

\begin{theorem}\label{lemma1} $\ $
\begin{enumerate}
\item  $\exists C_T>0$, such that for every $\mu,\nu \in \mathcal{M}_1^+\big( \C \times D \big)$,
\[
|\Gamma_{\nu}(\mu)-\Gamma(\mu)| \leq C_T \big(1+I(\mu|P)\big) d_T^V(\mu,\nu).
\]
\item If $\frac{2 \sigma^2 \bmax^2 T}{\lambda^2} < 1$, $H$ is a good rate function.
\end{enumerate}
\end{theorem}

\begin{proof}
The basic mechanism for the proof is similar than \cite[Lemma 4]{cabana-touboul:15} or \cite[Lemma 3.3-3.4]{benarous-guionnet:95}. However, the dependence in $x$ of the Gaussian $G^{\mu}(x)$ is problematic, as we cannot take it out of integrals on $x$. To cope with this difficulty, we will rely on tools from probability theory, such as Fubini's theorem for stochastic integrals, or Dambis-Dubins-Schwarz theorem (D.D.S.). We focus our attention on point 1., whereas point 2. previously shown without restriction on time in cases where $b(x,y)=S(y)$ \cite{cabana-touboul:12,cabana-touboul:15}, is now only valid under the short-time hypothesis of Proposition~\ref{prop:GammaBehaviour} point 2.

As proven in \cite{cabana-touboul:15}, $\Gamma_{\nu}$ writes $\Gamma_{\nu}(\mu)= \Gamma_{1,\nu}(\mu) + \Gamma_{2,\nu}(\mu)$ with
\begin{align*}
\Gamma_{1,\nu}(\mu) := - \frac{1}{2} \int_0^T \Big( \widetilde{K}^t_{\nu,x}(t,t)+m_{\nu}(t,x)^2 \Big) dt  d\mu(x,r),
\end{align*}
and
\begin{equation}\label{eq:Gamma2nu}
\Gamma_{2,\nu}(\mu) :=
\left\{
\begin{array}{cl}
{\frac{1}{2} \int_{{\C\times D}} \int_{\hat{\Omega}} L^{\nu}_T(x,r)^2  d\gamma_{\widetilde{K}_{\nu,x}^{T}} d\mu(x,r) + \int_{{\C\times D}} \int_0^T m_{\nu}(t,x)dW_t(x,r) d\mu(x,r)} & \text{if } I(\mu|P) < \infty,\\
+ \infty & \text{otherwise }.
\end{array}
\right.
\end{equation}

The previous decomposition has the interest of splitting the difficulties: $|\Gamma_{\nu}(\mu)-\Gamma(\mu)| \leq |\Gamma_{1,\nu}(\mu)-\Gamma_1(\mu)| + |\Gamma_{2,\nu}(\mu)-\Gamma_2(\mu)|$. 
The first term is easily controlled by $C_T d_T^V(\mu,\nu)$ using Proposition~\ref{prop:MeanAndVarRegularity}. Let us prove that 
\[
|\Gamma_{2,\nu}(\mu)-\Gamma_2(\mu)| \leq C_T (1+I(\mu|P))d_T^V(\mu,\nu).
\]
The inequality is trivial when $I(\mu \vert P)=\infty$. We now assume that $I(\mu \vert P)<\infty$ implying $\mu \ll P$ and finiteness of $\Gamma(\mu)$ and $\Gamma_{\nu}(\mu)$. In particular, $\mu$ has a Borel-measurable density $\rho_{\mu}$ with respect to $P$:
\[
d\mu(x,r)=\rho_{\mu}(x,r)dP(x,r).
\]

Let $\varepsilon>0$, and let $\xi \in \M_1^+\big( (\C \times D)^2 \big)$ with marginals $\mu$ and $\nu$ be such that 
\[
\bigg\{ \int_{(\C \times D)^2} d_T\big( (y,r'), (z,\tilde{r})\big)^2 d\xi\big( (y,r'),(z,\tilde{r})\big)  \bigg\}^{\frac{1}{2}} \leq d_T^V(\mu,\nu) +\varepsilon.
\]
Let also $\big(G(x), G'(x)\big)_{x \in \C}$ a family of bi-dimensional centered Gaussian process from the probability space $\big(\hat{\Omega},\hat{\mathcal{F}},\gamma\big)$ with covariance $K_{\xi}$ defined by \eqref{kxi}.
In the expression of $\Gamma_{2,\nu}(\mu)$ and $\Gamma_2(\mu)$ we can then replace the triplet $(G^{\mu}, G^{\nu}, \gamma)$ by $(G, G',\gamma)$, so that we choose their covariance to be given by $K_{\xi}$ (see \cite[Remark 3]{cabana-touboul:15}). As proved in Proposition~\ref{prop:MeanAndVarRegularity}, we can show that exist a constant $C_T>0$ such that for any $t \in [0,T]$, $x \in \C$,
\[
\E_{\gamma} \Big[ \big(G_t(x)-G'_t(x)\big)^2\Big] \leq \big(d_T^V(\mu,\nu) +\varepsilon\big)^2.
\]
Let also for any $t \in [0,T]$
\[
L_t(x,r):= \int_{0}^{t} G_s(x) dV^{\mu}_s(x,r), \quad L'_t(x,r):=\int_{0}^{t} G'_s(x) dV^{\nu}_s(x,r)
\]
Then,
\begin{align*}
& |\Gamma_{2,\nu}(\mu) - \Gamma_2(\mu)| \leq \frac{1}{2} \bigg|\int_{\C\times D}\E_{\gamma} \Big[ L'_T(x,r)^2 \big(\Lambda_T(G'(x))-\Lambda_T(G(x))\big)\Big] d\mu(x,r) \bigg| \\ 
& +\frac{1}{2} \bigg|\int_{\C\times D} \E_{\gamma} \Big[ \big( L_T(x,r)^2 - L'_T(x,r)^2 \big) \Lambda_T(G(x))\Big] d\mu(x,r) \bigg| + \bigg|\int_{\C\times D} \int_{0}^{T} (m_{\nu}-m_{\mu})(t,x)dW_t(x,r) d\mu(x,r)\bigg| 
 \end{align*}
Observe that by inequality \eqref{ineq:DiffLambda} we have
\begin{align*}
&\bigg|\int_{\C\times D}\E_{\gamma} \Big[ L'_T(x,r)^2 \big(\Lambda_T(G'(x))-\Lambda_T(G(x))\big)\Big] d\mu(x,r) \bigg| \leq C_T \Bigg\{ \big(d_T^V(\mu,\nu)+\varepsilon\big) \int_{\C\times D} \E_{\gamma} \Big[ L'_T(x,r)^2 \Big]  d\mu(x,r) \\
& +  \int_{\C\times D} \int_0^T \E_{\gamma} \bigg[\big| G_t(x)^2-{G'_t}(x)^2\big| L'_T(x,r)^2 \bigg] dt d\mu(x,r) \Bigg\} \overset{\mathrm{C.S}}{\leq} C_T \big(d_T^V(\mu,\nu)+\varepsilon\big) \int_{\C\times D} \E_{\gamma} \Big[ L'_T(x,r)^2 \Big]  d\mu(x,r),
\end{align*}
as Isserlis' theorem ensures that,
\begin{align*}
&\E_{\gamma}\Big[ \big(G'_t(x)-G_t(x)\big)^2 L'_T(x,r)^2 \Big] = \E_{\gamma}\Big[ \big(G'_t(x)-G_t(x)\big)^2 \Big] \E_{\gamma}\Big[ L'_T(x,r)^2 \Big] + 2 \E_{\gamma}\Big[ \big(G'_t(x)-G_t(x)\big)L'_T(x,r) \Big]^2 \\
& \overset{\mathrm{C.S}}{\leq} 3 \E_{\gamma}\Big[ \big(G'_t(x)-G_t(x)\big)^2 \Big] \E_{\gamma}\Big[ L'_T(x,r)^2 \Big] \leq 3 \big(d_T(\mu,\nu)+\varepsilon\big)^2 \E_{\gamma}\Big[ L'_T(x,r)^2 \Big],
\end{align*}
and similarly
\[
\E_{\gamma}\bigg[ \big(G'_t(x)+G_t(x)\big)^2 L'_T(x,r)^2 \bigg] \leq C_T \E_{\gamma}\Big[ L'_T(x,r)^2 \Big].
\]
As a consequence,
\begin{align}
&|\Gamma_{2,\nu}(\mu) - \Gamma_2(\mu)| \overset{\mathrm{C.S}}{\leq} C_T \Bigg\{ \overbrace{ \! \prod_{\varepsilon=\pm1} \Bigg( \int_{\C\times D} \!\E_{\gamma} \bigg[\Big( \int_0^T \big(G_t(x) +\varepsilon G'_t(x)\big) dV^{\nu}_t(x,r) \Big)^2 \bigg] d\mu(x,r) \Bigg)^\frac{1}{2}}^{B_1}   \nonumber  \\ 
& + \big(d_T^V(\mu,\nu)+\varepsilon\big) \underbrace{\int_{\C\times D} \E_{\gamma} \Big[ L'_T(x,r)^2 \Big]  d\mu(x,r)}_{B_2} + \underbrace{ \bigg( \int_{\C\times D}\Big| \int_0^T (m_{\nu}-m_{\mu})(t,x)dW_t(x,r) \Big|^2 d\mu(x,r) \bigg)^{\frac{1}{2}}}_{B_3} \nonumber\\
&+ \underbrace{ \prod_{\varepsilon=\pm1} \Bigg( \int_{\C\times D} \E_{\gamma} \bigg[\bigg\{ \int_0^T G_t(x) \Big( (1+\varepsilon)dW_t(x,r) - \big(m_{\mu}(t,x)+\varepsilon m_{\nu}(t,x)\big)dt \Big) \bigg\}^2 \bigg] d\mu(x,r) \Bigg)^{\frac{1}{2}}}_{B_4} \Bigg\}  \label{ineqgamma2}.
\end{align}

Remark that these four terms can be cast in the form
\[
\int_{\C\times D} \E_{\gamma} \bigg[ \Big(\int_0^T H_t(G,G',\mu,\nu)(x)\big(\alpha dW_t(x,r) - M_t(\mu,\nu)(x)dt \big) \Big)^2 \bigg] d\mu(x,r)
\]
with $\alpha$ equals $0$ or $1$. Controlling such terms is the aim of the following technical lemma.

\begin{lemma}\label{lemma3}
Let $\mu \in \M_1^+(\C \times D)$, with $\mu \ll P$ and let the filtration $\big(\mathcal{F}^{\mathbf{x}}_t\big)_{t \in [0,T]}$ on $\C$, where $\mathcal{F}^{\mathbf{x}}_t:= \sigma \big( x_s, 0 \leq s \leq t \big)$ is the $\sigma$-algebra on $\C$ generated by the coordinate process up to time $t$. Let also 
\begin{itemize}
\item $x \in \C \to \big(M_t (x) \big)_{t \in [0,T]}$ a bounded time-continuous process progressively measurable for the filtration  $(\mathcal{F}^{\mathbf{x}}_t)_{t \in [0,T]}$ and continuous in $x$,
\item $(x,\omega) \in \C \times \hat{\Omega} \to \Big(H_t(x,\omega) \Big)_{t \in [0,T]}$ a progressively measurable process for the filtration $\big(\mathcal{F}^{\mathbf{x}}_t \otimes \hat{\mathcal{F}} \big)_{t \in [0,T]}$, such that $\Big(H_t(x,\cdot), t \in [0,T] \Big)_{x \in \C}$ is a continuous family of $\gamma$-Gaussian processes (possibly deterministic) with uniformly bounded covariance,
\end{itemize}
and define
\[
A(\mu):= \int_{\C\times D} \int_{\hat{\Omega}} \Big(\int_0^T H_t(x,\omega)\big(\alpha dW_t(x,r) - M_t(x)dt \big) \Big)^2 d\gamma(\omega) d\mu(x,r)
\]
with $\alpha \in \{0,1\}$. 
Then, there exists a constant $C_T>0$ independent of $\mu$ such that
\begin{equation}
A(\mu) \leq C_T \bigg\{ \alpha \big( I(\mu|P) + 1 \big) + \underset{x \in \C, t \in [0,T]}{\sup} M_t^2(x) \bigg\} \underset{x \in \C, t \in [0,T]}{\sup} \E_{\gamma} \big[H_t^2(x)\big],
\end{equation}
with the right-hand side being possibly infinite.

\end{lemma}

\begin{proof}
As $(a+b)^2 \leq 2 a^2+ 2 b^2, \; \forall a,b \in \R$,
\begin{align*}
A(\mu,\nu) & \leq 2  \int_{\C\times D} \int_{\hat{\Omega}} \bigg\{ \alpha\Big( \underbrace{\int_0^T H_t(x,\omega) dW_t(x,r)}_{=:N_T(x,\omega,r)} \Big)^2 +\Big(\int_0^T H_t(x,\omega) M_t(x) dt \Big)^2 \bigg\} d\gamma(\omega) d\mu(x,r) \\
& \overset{\text{Fubini, C.S.}}{\leq} 2 \alpha \int_{\hat{\Omega}} \int_{\C\times D} N_T^2(x,\omega,r) d\mu(x,r) d\gamma(\omega) + 2 \int_{\C\times D} \int_0^T M_t^2(x) \E_{\gamma} \big[ H_t^2(x) \big] dt d\mu(x,r).
\end{align*}

Define the Radon-Nikodym density $\rho_{\mu}(x,r):=\der{\mu}{P}(x,r)$ and remark that for every $r \in D$, $\big(N_t (,\cdot,\cdot,r) \big)$ is, $\gamma$-a.s., a well-defined $P_r$-martingale. It\^o calculus gives, $\gamma$-a.s., the indistinguishable equality
\begin{equation}\label{eq:TechnicalLemmaIto}
N_T^2(x, \omega ,r) = 2 \int_0^T H_t(x,\omega) N_t(x,\omega,r)  dW_t(x,r) + \int_0^T H_t^2(x,\omega) dt,
\end{equation}
under $P_r$ so that, $\gamma$-a.s.,
\begin{align*}
\int_{\C\times D} N_T^2(x,\omega,r) & \rho_{\mu}(x,r)dP(x,r) = 2 \int_{\C\times D} \int_0^T H_t(x,\omega) N_t(x,\omega,r)  dW_t(x,r) \rho_{\mu}(x,r)dP(x,r)\\
&+ \int_{\C\times D} \int_0^T H_t^2(x,\omega) dt \rho_{\mu}(x,r)dP(x,r).
\end{align*}
Relying again on Fubini's Theorem,
\begin{align}
A(\mu,\nu) & \leq 4 \alpha \int_{\C\times D} \E_{\gamma} \bigg[ \int_0^T H_t(x) N_t(x,r) dW_t(x,r) \bigg]  \rho_{\mu}(x,r)dP(x,r) \nonumber\\
&+ 2 \int_{\C\times D} \int_0^T \E_{\gamma} \Big[ \alpha H_t^2(x) \Big] dt d\mu(x,r) + 2T \int_{\C\times D} \int_0^T M_t^2(x) \E_{\gamma} \big[H_t^2(x) \big] dt d\mu(x,r). \label{ineq:lemmaControlofTrickyTerms}
\end{align}
Under the favorable assumptions of the lemma, the last two terms of the right-hand side of \eqref{ineq:lemmaControlofTrickyTerms} are easily controlled taking the supremum of their integrand on $\C \times [0,T]$. In order to control the first term, we rely on the stochastic Fubini's theorem \cite[Theorem IV.65]{protter:05}, to show that the equality
\[
\tilde{N}_T(x,r):= \int_0^T \E_{\gamma} \bigg[  H_t(x) N_t(x,r) \bigg] dW_t(x,r) =\E_{\gamma} \bigg[ \int_0^T H_t(x) N_t(x,r) dW_t(x,r) \bigg],
\]
is well-defined, and holds $P$-almost surely. To do so, we need to ensure that:
\begin{enumerate}
\item $\forall r \in D$, $(x,\omega) \to \Big(\tilde{H}_t(x,\omega,r):= H_t(x,\omega)N_t(x,\omega,r )\Big)_{t \in [0,T]}$ is $\hat{\mathcal{F}} \otimes  \mathcal{P}$ measurable, where $\mathcal{P}$ is the $\sigma$-algebra generated by continuous $\big(\mathcal{F}^{\mathbf{x}}_t \big)_{t \in [0,T]}$-adapted processes,
\item the following integrability condition holds $\forall r \in D$:
\[
\int_{\C} \int_0^T \int_{\hat{\Omega}} \tilde{H}_t(x,\omega,r)^2 d\gamma(\omega) dt dP_r(x) < \infty.
\]
\end{enumerate}

The first hypothesis is a direct consequence of the regularity and measurability hypotheses of the lemma. We now demonstrate that the second hypothesis is valid. Indeed, for any $t \in [0,T]$,
\begin{align*}
& \int_{\C} \int_{\hat{\Omega}}  \tilde{H}_t(x,\omega,r)^2  d\gamma(\omega) dP_r(x) = \int_{\C} \E_{\gamma} \bigg[H_t(x,r)^2 N_t(x,r)^2 \bigg]  dP_r(x)\\
& \overset{C.S., \text{Fub.}}{\leq} \bigg\{\int_{\C} \E_{\gamma} \big[ H_t^4(x) \big] dP_r(x) \bigg\}^{\frac{1}{2}} \E_{\gamma} \bigg[ \int_{\C} N_t^4(x,r)  dP_r(x) \bigg]^{\frac{1}{2}} \overset{B.D.G.}{\leq} C_T \E_{\gamma} \bigg[ \int_{\C}  \langle N \rangle_t^2(x)  dP_r(x) \bigg]^{\frac{1}{2}} \\
&\overset{C.S., \text{Fubini}}{\leq} C_T \bigg\{ \int_{\C} \int_0^t \E_{\gamma} \Big[  H_s^4(x) \Big] ds dP_r(x) \bigg\}^{\frac{1}{2}}  < + \infty.
\end{align*}
Hence, the theorem applies so that 
\[
\int_{\C\times D} \E_{\gamma} \bigg[ \int_0^T H_t(x) N_t(x,r) dW_t(x,r) \bigg]  d\rho_{\mu}(x,r)dP(x,r) = \int_{\C\times D} \tilde{N}_T(x,r) d\mu(x,r).
\]
Observe that inequality (10) of \cite[p.12]{cabana-touboul:15} brings
\[
\int_{\C\times D} \! \! \! \! \! \! \tilde{N}_T(x,r) d\mu(x,r) \overset{\mathrm{C.S.}}{\leq} 2 \Big( \int_{\C\times D} \! \! \! \! \! \! \langle \tilde{N} \rangle_T(x,r) d\mu(x,r) \Big)^{\frac{1}{2}} \bigg(I(\mu|P) + \log\bigg\{\int_{\C\times D} \! \! \! \! \! \! \exp\bigg\{ \frac{\tilde{N}_T^2(x,r)}{4\langle \tilde{N} \rangle_T (x,r)} \bigg\} dP(x,r) \bigg\}  \bigg)^{\frac{1}{2}}.
\]
As $\tilde{N}(\cdot,r)$ is a $P_r$-local martingale for every $r \in D$, Dambis-Dubins-Schwarz (D.D.S.) theorem ensures that $\frac{\tilde{N}_T(\cdot,r)^2}{4\langle \tilde{N}(\cdot,r) \rangle_T } $ has the same law as $\frac{B_{\langle \tilde{N} \rangle_T}^2}{4 \langle \tilde{N} \rangle_T}$, where $B$ is some $P_r$-Brownian motion, so that exists a constant $C>0$ satisfying
\[
\log\bigg\{\int_{\C\times D} \exp\bigg\{ \frac{\tilde{N}_T^2(x,r)}{4\langle \tilde{N} \rangle_T(x,r) } \bigg\} dP(x,r) \bigg\} \leq C.
\]
We can therefore conclude that there exists two constants: $\tilde{C}>0$ independent of time, and $C_{T}>0$ increasing with $T$ such that:
\begin{align*}
& \int_{\C\times D} \tilde{N}_t(x,r) d\mu(x,r)  \leq \tilde{C} \bigg( \int_{\C\times D} \int_0^T \E_{\gamma} \bigg[ H_t(x) N_t(x,r) \bigg]^2 dt d\mu(x,r) \bigg)^{\frac{1}{2}} \big(I(\mu|P) + 1 \big)^{\frac{1}{2}}  \\
& \overset{\text{C.S., Fubini}}{\leq} 2\tilde{C} \; \underset{(x,t) \in \C \times [0,T]}{\sup} \bigg\{ \E_{\gamma} \Big[ \langle N \rangle_t(x) H_t^2(x) \Big] \bigg\}^{\frac{1}{2}}  \bigg( \int_0^T \E_{\gamma} \bigg[ \int_{\C\times D}  \frac{N_t^2(x,r)}{4 \langle N \rangle_t(x)} d\mu(x,r) \bigg] dt  \bigg)^{\frac{1}{2}} \big(I(\mu|P) + 1 \big)^{\frac{1}{2}} \\
& \overset{\text{D.D.S}}{\leq} C_T \underset{\C \times [0,T]}{\sup} \bigg\{ \E_{\gamma} \Big[ H_s^2(x) H_t^2(x) \Big] \bigg\}^{\frac{1}{2}}  \big(I(\mu|P) + 1 \big).
\end{align*}
where we have used inequality (10) of \cite[p.12]{cabana-touboul:15}. Isserlis' theorem then brings the result.
\end{proof}

It is easy to check that $B_1, \ldots, B_4$ are of the form of the terms handled in lemma~\eqref{lemma3}, satisfying in particular the adaptability conditions (we recall that the law of $G^{\nu}_t(x)$ depends on the trajectory of $x$ up to time $t$). To conclude, we underline that the quantities $\underset{x \in \C, t \in [0,T]}{\sup} \E_{\gamma} \Big[\big(G_t(x)-G'_t(x)\big)^2\Big]$, and $\underset{x \in \C, t \in [0,T]}{\sup} \big(m_{\mu}(t,x)-m_{\nu}(t,x)\big)^2$, are bounded by $\big(d_T^V(\mu,\nu)+\varepsilon\big)^2$ (see equation ~\eqref{ineq:MajDiff} for the term involving means).
\end{proof}

\subsection{Upper-bound and Tightness}

We are now in a position to demonstrate a partial LDP relying on the exponential tightness of the family $\Big(Q^N\big(\hat{\mu}_N \in \cdot \big) \Big)_N$, and an upper-bound inequality for closed subsets. To prove the first point, we rely on the exponential tightness of $P^{\otimes N}$ and the short time hypothesis \eqref{TimeCondition} and follow the approach proposed by Ben Arous and Guionnet in~\cite{ben-arous-guionnet:95,guionnet:97}. The second point is a consequence of an upper-bound for compact sets obtained similarly as in \cite[Theorem 7]{cabana-touboul:15}, and both the exponential tightness of $\Big(Q^N\big(\hat{\mu}_N \in \cdot \big) \Big)_N$ and the goodness of $H$ extending this bound to every closed sets.

\begin{theorem}\label{thm:WLDP}
	Under the condition $\frac{2 \sigma^2 \bmax^2 T}{\lambda^2} < 1$, we have:
\begin{enumerate}
\item For any real number $M \in \R$, there exists a compact set $K_M$ of $\M_1^+(\C \times D )$ such that, for any integer $N$,
\begin{equation*}
\frac{1}{N} \log Q^N(\hat{\mu}_N \notin K_M) \leq -M.
\end{equation*}
\item For any closed subset $F$ of $\M_1^+(\C \times D)$,
\begin{equation*}
\limsup_{N \rightarrow \infty} \frac{1}{N} \log{ Q^N(\hat{\mu}_N \in F)} \leq - \inf_{F} H .
\end{equation*}
\end{enumerate}
\end{theorem}

\begin{proof}

\noindent{(1):}

The proof of this theorem consists in using the exponential tightness of the sequence $(P^{\otimes N})_N$ provided e.g. by Sanov's Theorem and \cite[Exercice 1.2.18 (a)]{dembo-zeitouni:09}. Let $K_M$ be a compact of $\M_1^+(\C \times D)$ such that
\begin{equation*}
\frac{1}{N} \log P^{\otimes N}(\hat{\mu}_N \notin K_M) \leq -M,
\end{equation*}
and remark that H\"older inequality yields for any conjugate exponents $(p,q)$ with $\frac{(p+1)p^2\sigma^2 \bmax^2 T}{\lambda^2}<1$:
\begin{align*}
Q^N(\hat{\mu}_N \not\in K_M) & \leq \bigg( \int_{(\C \times D)^N} \exp\big\{p N\bar{\Gamma}(\hat{\mu}_N)  \big\} dP^{\otimes N}(\mathbf{x},\mathbf{r}) \bigg)^{\frac{1}{p}} P^{\otimes N}(\hat{\mu}_N \not\in K_M)^{\frac{1}{q}} \\
& \overset{\text{Jensen}}{\leq} \bigg( \int_{(\C \times D)^N} \prod_{i=1}^N \E_{\gamma} \bigg( \exp\big\{ p X^{\hat{\mu}_N}(x^i,r_i) \big\} \bigg) dP^{\otimes N}(\mathbf{x},\mathbf{r}) \bigg)^{\frac{1}{p}} P^{\otimes N}(\hat{\mu}_N \not\in K_M)^{\frac{1}{q}}
\end{align*}
Let $(\tilde{X}^{\hat{\mu}_N,i})_{1 \leq i \leq N}$ be independent copies of $X^{\hat{\mu}_N}$ under the measure $\gamma$. Then, by independence, H\"older inequality and martingale property, we have
\begin{align}
\int_{(\C \times D)^N} & \prod_{i=1}^N \E_{\gamma} \bigg( \exp\big\{ p X^{\hat{\mu}_N}(x^i,r_i) \big\} \bigg) dP^{\otimes N}(\mathbf{x},\mathbf{r}) =  \E_{\gamma} \bigg[ \int_{(\C \times D)^N} \exp\bigg\{ p \sum_{i=1}^N  \tilde{X}^{\hat{\mu}_N,i}(x^i,r_i) \bigg\} dP^{\otimes N}(\mathbf{x},\mathbf{r}) \bigg] \nonumber\\
& \leq \Bigg( \int_{(\C \times D)^N} \prod_{i=1}^N \E_{\gamma} \bigg[ \exp\big\{ \frac{p^2(p+1)}{2} \int_0^T \big(G^{\hat{\mu}_N}_t(x^i)+m_{\hat{\mu}_N}(t,x^i)\big)^2 dt \big\} \bigg] dP^{\otimes N}(\mathbf{x},\mathbf{r}) \Bigg)^{\frac{p-1}{p}}. \label{ineq:calculus1}
\end{align}
We can now proceed as in the proof of Proposition~\ref{prop:GammaBehaviour}. point 2. to find that exists a constant $c_T$ such that
\[
\int_{(\C \times D)^N} \prod_{i=1}^N \E_{\gamma} \bigg( \exp\big\{ p X^{\hat{\mu}_N}(x^i,r_i) \big\} \bigg) dP^{\otimes N}(\mathbf{x},\mathbf{r}) \leq e^{(p-1) c_T N}.
\]
As a consequence,
\begin{align*}
\limsup_{N \to +\infty} \frac{1}{N} \log Q^N(\hat{\mu}_N \not\in K_M) \leq (p-1)c_T - \frac{M}{q}.
\end{align*}

\noindent{(2):}
As $\big(Q^N(\hat{\mu}_N \in \cdot)\big)_N$ is exponentially tight and $H$ is good, it is sufficient to prove the upper-bound for compact sets (see \cite[Lemma 1.2.18 (a)]{dembo-zeitouni:09}). We obtain this upper-bound exactly as in \cite[Theorem 7]{cabana-touboul:15} relying on the following lemma.
\end{proof}

\begin{lemma}\label{lemma3.1}
For any real number $q >1$, if $\frac{2\sigma^2 \bmax^2 T}{\lambda^2} < 1$, then exist a strictly positive real number $\delta_q$ and a function $C_q:\R^+ \to \R^+$ such that $\lim_{\delta \rightarrow 0} C_q(\delta) =0$, and for any $\delta < \delta_q$:
\begin{equation*}
\int_{\hat{\mu}_N \in K \cap B(\nu,\delta)} \E_{\gamma} \Bigg[ \prod_{i=1}^N \bigg( \exp{ q \Big( \tilde{X}^{\hat{\mu}_N,i}(x^i,r_i) - \tilde{X}^{\nu,i}(x^i,r_i) \Big) } \exp{\tilde{X}^{\nu,i}(x^i,r_i)} \bigg)  \Bigg]  dP^{\otimes N}(\mathbf{x},\mathbf{r}) \leq \exp\{C_q(\delta) N\}.
\end{equation*}
\end{lemma}

\begin{proof}
Let
\[
B_N:=\int_{\hat{\mu}_N \in K \cap B(\nu,\delta)} \E_{\gamma} \Bigg[ \prod_{i=1}^N \bigg( \exp{ q \Big( \tilde{X}^{\hat{\mu}_N,i}(x^i,r_i) - \tilde{X}^{\nu,i}(x^i,r_i) \Big) } \exp{\tilde{X}^{\nu,i}(x^i,r_i)} \bigg)  \Bigg]  dP^{\otimes N}(\mathbf{x},\mathbf{r}).
\]
We again split this quantity relying on H\"older inequality with conjugate exponents $(\rho, \eta)$:
\begin{multline}
B_N  \leq  \Bigg\{ \overbrace{\int_{(\C \times D)^N} \prod_{i=1}^N \E_{\gamma} \Bigg[   \exp{ \rho X^{\nu}(x^i,r_i)} \Bigg] dP^{\otimes N}(\mathbf{x},\mathbf{r})}^{B^N_1} \Bigg\}^{\frac{1}{\rho}} \\
\times \Bigg\{ \underbrace{\int_{\hat{\mu}_N \in B(\nu,\delta)} \E_{\gamma} \Bigg[ \prod_{i=1}^N \exp{ q \eta \Big( \tilde{X}^{\hat{\mu}_N,i}(x^i,r_i) - \tilde{X}^{\nu,i}(x^i,r_i) \Big) }   \Bigg] dP^{\otimes N}(\mathbf{x},\mathbf{r})}_{B^N_2} \Bigg\}^{\frac{1}{\eta}}. \label{ineqlemma}
\end{multline}

On the one hand, we can proceed exactly as in calculus \eqref{ineq:calculus1} to obtain the existence of a constant $c_T$ uniform in $\rho$ and $N$ such that
\begin{align*}
B^N_1 & \leq e^{N(\rho-1) c_T},
\end{align*}
so that one has to choose the proper relation between $\rho-1$ and $\delta$. On the other hand, the second term can be handled exactly as in \cite[Lemma 5]{cabana-touboul:15}.
\end{proof}

\section{Existence and characterization of the limit}\label{sec:LimitIdentification}

\subsection{Uniqueness of the minimum}

This section is devoted to prove existence and uniqueness of the minima of $H$ in order to obtain exponential convergence of the empirical measure. At this end, we proceed as in \cite[Lemma 6]{cabana-touboul:15} to obtain a convenient characterization of the minima of $H$:

\begin{lemma}\label{Qcharac}
	Let $\mu$ be a probability measure on $\C \times D$ which minimizes $H$. Then
\begin{equation}\label{eq:densityMinimum}
	\mu \simeq P, \qquad \mu = Q_{\mu},
\end{equation}
where $\mu \to Q_{\mu}$ introduced in \eqref{def:Qnu} is well-defined from $\M_1^+(\C \times D) \to \M_1^+(\C \times D)$.
\end{lemma}

We then have:

\begin{theorem}\label{thm:UniquenessOfTheMinimum}
The map $\mu \to Q_{\mu}$ admits a unique fixed point.
\end{theorem}

\begin{proof}
As in \cite[Lemma 3]{cabana-touboul:15} or \cite[Lemma 5.15]{ben-arous-guionnet:95}, we can show that
\begin{align*}
\der{Q_{\mu}}{P}(x,r) = \exp\bigg\{\int_0^T O_{\mu}(t,x,r)dW_t(x,r)-\frac{1}{2}\int_0^T O_{\mu}^2(t,x,r)dt\bigg\},
\end{align*}
where 
\begin{align*}
O_{\mu}(t,x,r)& =  \E_{\gamma} \Big[ \Lambda_t\big( G^{\mu}(x) \big) G^{\mu}_t(x) L^{\mu}_t(x,r) \Big]+ m_{\mu}(t,x).
\end{align*}
Moreover, as done in \cite[Theorem 6]{cabana-touboul:15}, we introduce $Q_{\mu,r} \in \M_1^+(\C)$ such as $dQ_{\mu}(x,r)=dQ_{\mu,r}(x)d\pi(r)$ for every $(x,r)\in \C \times D$. Girsanov's theorem naturally leads to introduce the following SDE whose putative solution have a law equal to $Q_{\mu,r}$:
\begin{equation}\label{SDE:uniqueness}
\left\{
  \begin{array}{ll}
  dx_t^{{\mu}}(r) = f(r,t,x_t^{{\mu}}(r)) dt + \lambda O^{\tilde{W}}_{\mu}(t,x_t^{{\mu}}(r)) dt + \lambda d\tilde{W}_t\\
    x_0^{{\mu}}(r) = \bar{x}_0(r).
  \end{array}
\right.
\end{equation}
where $\tilde{W}$ is a $\mathbbm{P}$-Brownian motion, where
\[
O^{\tilde{W}}_{\mu}(t,x):= \E_{\gamma} \Big[ \Lambda_t\big( G^{\mu}(x) \big) G^{\mu}_t(x) \tilde{L}^{\mu}_t(x) \Big] + m_{\mu}(t,x),
\]
\[
\tilde{L}^{\mu}_t(x):= \int_0^t G^{\mu}_s(x) \big(d\tilde{W}_s- m_{\mu}(s,x)ds\big),
\]
and $\bar{x}_0(r) \in \R$ is the realization of the continuous version for the family of initial laws $\big(\mu_0(r)\big)_{r \in D}$ evaluated at $r$ (see hypothesis (3) of \cite[p.6]{cabana-touboul:15}). We show in Lemma~\ref{lem:existence_uniqueness} (see below) that for any $(r,\mu) \in D\times \M_{1}^{+}(\C\times D)$, there exists a unique strong solution $(x_t^{{\mu}}(r))_{t\in [0,T]}$ to equation~\eqref{SDE:uniqueness}. Let $\nu \in \M_1^+\big( \C \times D \big)$, and define similarly $x^{\nu}_t(r)$ with same initial condition and Brownian path. Then
\begin{align}
& \big( x^{\mu}_t(r) - x^{\nu}_t(r) \big) = \int_0^t \Big(f(r,s,x^{\mu}_{s}(r))+ \lambda m_{\mu}(s,x^{\mu}_{{\cdot}}(r))- f(r,s,x^{\nu}_s(r))-\lambda m_{\nu}(s,x^{\nu}_{{\cdot}}(r)\Big) ds \nonumber\\
& + \lambda \int_0^t \bigg\{\E_{\gamma} \Big[ \Lambda_s\big( G^{\mu}(x^{\mu}_{\cdot}(r)) \big) G^{\mu}_s(x^{\mu}_{\cdot}(r)) \tilde{L}^{\mu}_s(x^{\mu}_{\cdot}(r))\Big] - \E_{\gamma} \Big[ \Lambda_s\big( G^{\nu}(x^{\mu}_{\cdot}(r)) \big) G^{\nu}_s(x^{\mu}_{\cdot}(r)) \tilde{L}^{\nu}_s(x^{\mu}_{\cdot}(r)) \Big]\bigg\} ds \nonumber\\
& + \lambda \int_0^t \int_0^s \bigg(\tilde{K}^{s}_{\nu,x^{\mu}_{\cdot}(r)}(s,v) m_{\nu}\big(v,x^{\mu}_{\cdot}(r)\big) -\tilde{K}^{s}_{\nu,x^{\nu}_{\cdot}(r)}(s,v) m_{\nu}\big(v,x^{\nu}_{\cdot}(r)\big) \bigg)dv ds \nonumber \\
& + \lambda \int_0^t \bigg\{\E_{\gamma} \bigg[ \Lambda_s\big( G^{\nu}(x^{\mu}_{\cdot}(r)) \big) G^{\nu}_s(x^{\mu}_{\cdot}(r)) \Big(\int_0^s G^{\nu}_v(x^{\mu}_{\cdot}(r)) d\tilde{W}_v \Big)\bigg] \nonumber\\
& - \E_{\gamma} \bigg[ \Lambda_s\big( G^{\nu}(x^{\nu}_{\cdot}(r)) \big) G^{\nu}_s(x^{\nu}_{\cdot}(r)) \Big(\int_0^s G^{\nu}_v(x^{\nu}_{\cdot}(r)) d\tilde{W}_v \Big) \bigg]\bigg\} ds  \label{ineq:FixPointAverage}.
\end{align}

Let $\xi \in \M_1^+\big((\C \times D)^2\big)$ with marginals $\mu$ and $\nu$. We have:
\begin{align}\label{ineq:uniquenessth1}
& \lambda \big(m_{\mu}(t,x^{\mu}_{\cdot}(r))- m_{\nu}(t,x^{\nu}_{\cdot}(r))\big) = \bar{J} \int_{(\C \times D)^{2}} \Big(b(x^{\mu}_t(r),y_t)- b(x^{\nu}_t(r),y_t)\Big) \nonumber \\
& + \Big(b(x^{\nu}_t(r),y_t)- b(x^{\nu}_t(r),z_t)\Big) d\xi\big((y,r'),(z,\tilde{r})\big) \nonumber\\
& \leq K_b \bar{J} \Big(\big| x^{\mu}_t(r)-x^{\nu}_t(r)\big| + \int_{(\C \times D)^{2}} \Vert y-z \Vert_{\infty,t} d\xi\big((y,r'),(z,\tilde{r})\big)\Big) \leq C \Big(\big| x^{\mu}_t(r)-x^{\nu}_t(r)\big| + d_t^V(\mu,\nu) \Big)
\end{align}
where we took the infimum on $\xi$.

Furthermore, let $\big(\tilde{G},\tilde{G}'\big)$ be a $\gamma$-bidimensional centered Gaussian process with covariance given by:
\begin{equation}\label{def:Kximunu}
\frac{\sigma^2}{\lambda^2} \int_{(\C \times D)^2} 
\left(
\begin{array}{ccc}
b(x^{\mu}_s(r),y_s)b(x^{\mu}_t(r),y_t)  & b(x^{\mu}_s(r),y_s)b(x^{\nu}_t(r),y_t) \\
b(x^{\nu}_s(r),y_s)b(x^{\mu}_t(r),y_t) & b(x^{\nu}_s(r),y_s)b(x^{\nu}_t(r),y_t) \\
\end{array}
\right) 
d\nu(y,r).
\end{equation}
Then

\begin{align*}
\Big| \tilde{K}^{t}_{\nu,x^{\mu}_{\cdot}(r)}(t,s) -\tilde{K}^{t}_{\nu,x^{\nu}_{\cdot}(r)}(t,s)\Big| = \Big|\E_{\gamma} \Big[ \big(\Lambda_t(\tilde{G}) -\Lambda_t(\tilde{G}')\big) \tilde{G}_t \tilde{G}_s  + \Lambda_t(\tilde{G}') \big(\tilde{G}_t - \tilde{G}'_t \big)\tilde{G}_s + \Lambda_t(\tilde{G}')\tilde{G}'_t \big( \tilde{G}_s - \tilde{G}'_s\big) \Big] \Big|
\end{align*}
Moreover,
\begin{align}\label{ineq:UniquenessIneq1}
\E_{\gamma} \Big[ \Lambda_t( \tilde{G}) \big(\tilde{G}_t -\tilde{G}'_t \big) \tilde{G}_s \Big] & \overset{\mathrm{C.S.},\eqref{ineq:BoundOnLambda}}{\leq} C_T \E_{\gamma} \Big[ \big(\tilde{G}_t -\tilde{G}'_t \big)^2 \Big]^{\frac{1}{2}} \leq C_T \left( \int_{\C \times D} \Big( b\big(x^{\mu}_t(r),y_t\big) -b\big(x^{\nu}_t(r),y_t\big) \Big)^2 d\nu(y,r') \right)^{\frac{1}{2}} \nonumber\\
& \leq C_T \big| x^{\mu}_t(r)-x^{\nu}_t(r) \big|
\end{align}
and
\begin{align*}
& \E_{\gamma} \Big[ \big(\Lambda_t( \tilde{G})-\Lambda_t( \tilde{G}')\big) \tilde{G}_t \tilde{G}_s \Big]
\overset{\eqref{ineq:DiffLambda}, \mathrm{C.S.}}{\leq} C_T \big| x^{\mu}_t(r)-x^{\nu}_t(r) \big| ,
\end{align*}
so that 
\begin{equation}\label{ineq:uniquenessth2}
\Big| \tilde{K}^{t}_{\nu,x^{\mu}_{\cdot}(r)}(t,s) -\tilde{K}^{t}_{\nu,x^{\nu}_{\cdot}(r)}(t,s)\Big| \leq C_T \big| x^{\mu}_t(r)-x^{\nu}_t(r) \big|.
\end{equation}

We now focus on controlling the second term of \eqref{ineq:FixPointAverage}. Let another $\xi \in \M_1^+\big((\C \times D)^2\big)$ with marginals $\mu$ and $\nu$, and the couple $\big(G,G'\big)$ of centered $\gamma$-Gaussian process with covariance $K_{\xi}\big(\cdot, \cdot,x^{\mu}_{\cdot}(r)\big)$ given in \eqref{kxi}. Replacing the couple $\big(G^{\mu}(x^{\mu}_{\cdot}(r)),G^{\nu}(x^{\mu}_{\cdot}(r))\big)$ by $(G,G')$ in the term of interest, we obtain:
\begin{align*}
\E_{\gamma} \Big[ \Lambda_s\big( G^{\mu}(x^{\mu}_{\cdot}(r)) \big) G^{\mu}_s(x^{\mu}_{\cdot}(r)) \tilde{L}^{\mu}_s(x^{\mu}_{\cdot}(r))\Big] & - \E_{\gamma}\Big[\Lambda_s\big( G^{\nu}(x^{\mu}_{\cdot}(r)) \big) G^{\nu}_s(x^{\mu}_{\cdot}(r)) \tilde{L}^{\nu}_s(x^{\mu}_{\cdot}(r)) \Big] \\
& = \E_{\gamma} \Big[ \Lambda_s(G) G_s L_s - \Lambda_s( G') G'_s L'_s \Big]
\end{align*}
where $L_t:= \int_0^t G_s \big(d\tilde{W}_s- m_{\mu}(s,x^{\mu}_{\cdot}(r))ds\big)$, and $L'_t:=\int_0^t G'_s \big(d\tilde{W}_s- m_{\nu}(s,x^{\mu}_{\cdot}(r))ds\big)$.

Moreover,
\begin{align*}
& \E_{\gamma} \Big[ \Lambda_t(G) G_t L_t - \Lambda_t(G') G'_t L'_t \Big]= \E_{\gamma} \Big[ \big(\Lambda_t(G) -\Lambda_t( G') \big) G_t L_t \Big] + \E_{\gamma} \Big[ \Lambda_t( G') \big(G_t -G'_t \big) L_t \Big] \\
& + \E_{\gamma} \Big[ \Lambda_t( G')  G'_t \big( L_t - L'_t\big) \Big] \overset{\mathrm{C.S.}}{\leq} \E_{\gamma} \big[L_t^2 \big]^{\frac{1}{2}} \bigg( \E_{\gamma} \Big[ \big(\Lambda_t( G ) -\Lambda_t( G') \big)^2 G_t^2 \Big]^{\frac{1}{2}}+ \E_{\gamma} \Big[ \Lambda_t( G')^2 \big(G_t -G'_t\big)^2 \Big]^{\frac{1}{2}}\bigg) \\
& + \E_{\gamma} \Big[ \Lambda_t( G')^2  {G'_t}^2 \Big]^{\frac{1}{2}} \E_{\gamma} \Big[\big( L_t - L'_t\big)^2\Big]^{\frac{1}{2}}
\end{align*}

On the one hand, we can show as in Proposition~\ref{prop:MeanAndVarRegularity} that
\begin{align*}
\E_{\gamma} \Big[ \Lambda_t( G)^2 \big(G_t -G'_t \big)^2 \Big] \leq C_T \left( \int_{(\C \times D)^{2}} \Vert y-z \Vert_{\infty,t}^2 d\xi\big((y,r'),(z,\tilde{r})\big) \right),
\end{align*}
and by Isserlis' theorem,
\begin{align*}
& \E_{\gamma} \Big[ \big(\Lambda_t( G)-\Lambda_t( G')\big)^2 G_t^2 \Big] \overset{\eqref{ineq:DiffLambda}}{\leq} C_T \left(\int_{(\C \times D)^{2}} \Vert y-z \Vert_{\infty,t}^2 d\xi\big((y,r'),(z,\tilde{r})\big) \right).
\end{align*}
On the other hand, 
\begin{align*}
& \E_{\gamma} \Big[\big( L_t - L'_t\big)^2\Big] \leq 2  \E_{\gamma} \bigg[\bigg(\int_0^t \big(G_s-G'_s\big) d\tilde{W}_s \bigg)^2\bigg] +4t \int_0^t \bigg\{ \E_{\gamma} \bigg[ \big(G_s-G'_s\big)^2 m_{\mu}(s,x^{\mu}_{\cdot}(r))^2 \bigg] \\
& + \E_{\gamma} \bigg[ {G'_s}^2 \Big(m_{\mu}\big(s,x^{\mu}_{\cdot}(r)\big) - m_{\nu}\big(s,x^{\mu}_{\cdot}(r)\big)\Big)^2  \bigg]\bigg\}ds \\
& \overset{\eqref{ineq:MajDiff}}{\leq} C_T \Bigg\{ \E_{\gamma} \bigg[\bigg(\int_0^t \big(G_s-G'_s\big) d\tilde{W}_s \bigg)^2\bigg] + \int_{(\C \times D)^{2}} \Vert y-z \Vert_{\infty,t}^2 d\xi\big((y,r'),(z,\tilde{r})\big)  \Bigg\},
\end{align*}
eventually yielding:
\begin{align*}
& \bigg|\E_{\gamma} \Big[ \Lambda_s\big( G^{\mu}(x^{\mu}_{\cdot}(r)) \big) G^{\mu}_s(x^{\mu}_{\cdot}(r)) \tilde{L}^{\mu}_s(x^{\mu}_{\cdot}(r))\Big] - \E_{\gamma}\Big[\Lambda_s\big( G^{\nu}(x^{\mu}_{\cdot}(r)) \big) G^{\nu}_s(x^{\mu}_{\cdot}(r)) \tilde{L}^{\nu}_s(x^{\mu}_{\cdot}(r)) \Big]\bigg|^2\\
& \leq C_T \bigg( 1+ \E_{\gamma} \bigg[ \Big(\int_0^s G_v d\tilde{W}_v \Big)^2 \bigg] \bigg)\left(\int_{(\C \times D)^{2}} \Vert y-z \Vert_{\infty,s}^2 d\xi\big((y,r'),(z,\tilde{r})\big)  \right) + C_T \E_{\gamma} \bigg[\bigg(\int_0^s \big(G_v-G'_v\big) d\tilde{W}_v \bigg)^2\bigg].
\end{align*}
As a consequence, equation \eqref{ineq:FixPointAverage} becomes:
\begin{align*}
& \Vert x^{\mu}(r) - x^{\nu}(r) \Vert_{\infty,t}^2 \leq C_T \int_0^t \Bigg\{ \Vert x^{\mu}(r)-x^{\nu}(r) \Vert_{\infty,s}^2 + \E_{\gamma} \bigg[\bigg(\int_0^s \big(G_v-G'_v\big) d\tilde{W}_v \bigg)^2\bigg]\\
& + \bigg( 1+ \E_{\gamma} \bigg[ \Big(\int_0^s G_v d\tilde{W}_v \Big)^2 \bigg] \bigg)\left(\int_{(\C \times D)^{2}} \Vert y-z \Vert_{\infty,s}^2 d\xi\big((y,r'),(z,\tilde{r})\big) \right)\\
& + \bigg| \E_{\gamma} \bigg[ \Lambda_s\big( G^{\nu}(x^{\mu}_{\cdot}(r)) \big) G^{\nu}_s(x^{\mu}_{\cdot}(r)) \Big(\int_0^s G^{\nu}_v(x^{\mu}_{\cdot}(r)) d\tilde{W}_v \Big)\bigg] \\
& - \E_{\gamma} \bigg[ \Lambda_s\big( G^{\nu}(x^{\nu}_{\cdot}(r)) \big) G^{\nu}_s(x^{\nu}_{\cdot}(r)) \Big(\int_0^s G^{\nu}_v(x^{\nu}_{\cdot}(r)) d\tilde{W}_v \Big) \bigg]\bigg|^2 \Bigg\} ds.
\end{align*}
Relying on Gronwall's lemma, taking the expectation over both initial conditions and the Brownian path, making use of Fubini's theorem, It\^o isometry, and eventually taking the infimum in $\xi$ yields:
\begin{align*}
& \Exp \bigg[ \Vert x^{\mu}(r) - x^{\nu}(r) \Vert_{\infty,t}^2 \bigg] \leq C_T \int_0^t \Bigg\{ \left(\int_{(\C \times D)^{2}} \Vert y-z \Vert_{\infty,s}^2 d\xi\big((y,r'),(z,\tilde{r})\big) \right)\\
+ & \Exp \Bigg[ \bigg| \E_{\gamma} \bigg[ \Lambda_s\big( G^{\nu}(x^{\mu}_{\cdot}(r)) \big) G^{\nu}_s(x^{\mu}_{\cdot}(r)) \Big(\int_0^s G^{\nu}_v(x^{\mu}_{\cdot}(r)) d\tilde{W}_v \Big)\bigg] \\
& - \E_{\gamma} \bigg[ \Lambda_s\big( G^{\nu}(x^{\nu}_{\cdot}(r)) \big) G^{\nu}_s(x^{\nu}_{\cdot}(r)) \Big(\int_0^s G^{\nu}_v(x^{\nu}_{\cdot}(r)) d\tilde{W}_v \Big) \bigg]\bigg|^2 \Bigg] \Bigg\} ds.
\end{align*}

To cope with the last term of the right-hand side, let again $\big(\tilde{G},\tilde{G}'\big)$ be a bidimensional centered Gaussian process on the probability space $\big(\hat{\Omega}, \hat{\mathcal{F}}, \gamma \big)$ with covariance given by~\eqref{def:Kximunu}. Let also $\Exp_{\gamma} \big[ \cdot \big] := \Exp\Big[ \E_{\gamma} \big[ \cdot \big] \Big]$. Then 
\begin{align}
& \Exp \Bigg[ \bigg| \E_{\gamma} \bigg[ \Lambda_s\big( G^{\nu}(x^{\mu}_{\cdot}(r)) \big) G^{\nu}_s(x^{\mu}_{\cdot}(r)) \Big(\int_0^s G^{\nu}_v(x^{\mu}_{\cdot}(r)) d\tilde{W}_v \Big)\bigg] \nonumber\\
& - \E_{\gamma} \bigg[ \Lambda_s\big( G^{\nu}(x^{\nu}_{\cdot}(r)) \big) G^{\nu}_s(x^{\nu}_{\cdot}(r)) \Big(\int_0^s G^{\nu}_v(x^{\nu}_{\cdot}(r)) d\tilde{W}_v \Big) \bigg]\bigg|^2 \Bigg] \nonumber \\
& = \Exp \Bigg[ \E_{\gamma} \bigg[ \Lambda_s( \tilde{G} ) \tilde{G}_s \Big( \int_0^s \tilde{G}_v d\tilde{W}_v \Big) - \Lambda_s( \tilde{G}') \tilde{G}'_s \Big( \int_0^s \tilde{G}'_v d\tilde{W}_v \Big) \bigg]^2 \Bigg] \nonumber \\
& \overset{\mathrm{C.S}}{\leq} \Exp_{\gamma} \Bigg[ \bigg\{ \Lambda_s( \tilde{G} ) \tilde{G}_s \Big( \int_0^s \tilde{G}_v d\tilde{W}_v \Big) - \Lambda_s( \tilde{G}') \tilde{G}'_s \Big( \int_0^s \tilde{G}'_v d\tilde{W}_v \Big) \bigg\}^2 \Bigg] \nonumber\\
& \overset{\mathrm{C.S}}{\leq} 3\Exp_{\gamma} \bigg[\Big( \int_0^s \tilde{G}_v d\tilde{W}_v \Big)^4 \bigg]^{\frac{1}{2}} \bigg( \Exp_{\gamma} \Big[ \big(\Lambda_t( \tilde{G} ) -\Lambda_t( \tilde{G}') \big)^4 \tilde{G}_t^4 \Big]^{\frac{1}{2}}+ \Exp_{\gamma} \Big[ \Lambda_t( \tilde{G}')^4 \big(\tilde{G}_t -\tilde{G}'_t\big)^4 \Big]^{\frac{1}{2}}\bigg)  \nonumber\\
& + 3\Exp_{\gamma} \Big[ \Lambda_t( \tilde{G}')^4  {\tilde{G}'_t}^4 \Big]^{\frac{1}{2}} \Exp_{\gamma} \bigg[\Big( \int_0^s \big(\tilde{G}_v-\tilde{G}'_v\big) d\tilde{W}_v \Big)^4\bigg]^{\frac{1}{2}}. \label{ineq:UniquenessDifficultTerm}
\end{align}
Gaussian calculus and \eqref{ineq:UniquenessIneq1} gives
\[
\E_{\gamma} \Big[ \big(\tilde{G}_t -\tilde{G}'_t\big)^4 \Big] = C \E_{\gamma} \Big[ \big(\tilde{G}_t -\tilde{G}'_t\big)^2 \Big]^2 \leq C_T \big| x^{\mu}_t(r)-x^{\nu}_t(r) \big|^2.
\]
Then relying on \eqref{ineq:BoundOnLambda}, \eqref{ineq:DiffLambda} and Burkh\"older Davis Gundi inequality, we obtain:
\begin{align*}
\Exp \bigg[ \Vert x^{\mu}(r) - x^{\nu}(r) \Vert_{\infty,t}^2 \bigg] & \leq C_T \int_0^t \bigg\{ \left(\int_{(\C \times D)^{2}} \Vert y-z \Vert_{\infty,s}^2 d\xi\big((y,r'),(z,\tilde{r})\big) \right) + \Exp \bigg[ \Vert x^{\mu}(r) - x^{\nu}(r) \Vert_{\infty,s}^2 \bigg] \bigg\} ds.
\end{align*}
Another use of Gronwall's lemma then yields for any $\xi \in \M_1^+\big(( \C \times D)^2 \big)$ with marginals $\mu$ and $\nu$:
\begin{align}\label{ineq:uniquenessth3}
\Exp \bigg[ \Vert x^{\mu}(r) - x^{\nu}(r) \Vert_{\infty,t}^2 \bigg] & \leq C_T \int_0^t \left(\int_{(\C \times D)^{2}} \Vert y-z \Vert_{\infty,s}^2 d\xi\big((y,r'),(z,\tilde{r})\big) \right) ds.
\end{align}

Let us prove the regularity in space of left-hand side in the above inequality. In fact, fix $r' \neq r \in D$, and consider $x^{\mu}_\cdot{(r')}$ be the strong solution of \eqref{SDE:uniqueness} with same $\tilde{W}$ but initial condition given by $\bar{x}_0(r')$. We have for any $t \in [0,T]$,
\begin{align*}
& x^{\mu}_t(r) - x^{\mu}_t(r') = \big(\bar{x}_0(r)- \bar{x}_0(r')\big) + \int_0^t \Big(f(r,s,x^{\mu}_{s}(r))+ \lambda m_{\mu}(s,x^{\mu}_{{\cdot}}(r))- f(r',s,x^{\mu}_s(r'))-\lambda m_{\mu}(s,x^{\mu}_{{\cdot}}(r)'\Big) ds\\
& + \lambda \int_0^t \int_0^s \bigg(\tilde{K}^{s}_{\mu,x^{\mu}_{\cdot}(r)}(s,v) m_{\mu}\big(v,x^{\mu}_{\cdot}(r)\big) -\tilde{K}^{s}_{\mu,x^{\mu}_{\cdot}(r')}(s,v) m_{\mu}\big(v,x^{\mu}_{\cdot}(r')\big) \bigg)dv ds \\
& + \lambda \int_0^t \bigg\{\E_{\gamma} \bigg[ \Lambda_s\big( G^{\mu}(x^{\mu}_{\cdot}(r)) \big) G^{\mu}_s(x^{\mu}_{\cdot}(r)) \Big(\int_0^s G^{\mu}_v(x^{\mu}_{\cdot}(r)) d\tilde{W}_v \Big)\bigg] \\
& - \E_{\gamma} \bigg[ \Lambda_s\big( G^{\mu}(x^{\mu}_{\cdot}(r')) \big) G^{\mu}_s(x^{\mu}_{\cdot}(r')) \Big(\int_0^s G^{\mu}_v(x^{\mu}_{\cdot}(r')) d\tilde{W}_v \Big) \bigg]\bigg\} ds.
\end{align*}

Then, developing a similar analysis than above, we find:
\begin{align*}
\Exp \Big[ \Vert  x^{\mu}(r) - x^{\mu}(r') \Vert_{\infty,t}^2\Big] & \leq C_T \bigg\{\Exp \Big[\big(\bar{x}_0(r)- \bar{x}_0(r') \big)^2 \Big]+ \Vert r-r' \Vert_{\R^d}^2\bigg\},
\end{align*}
so that $\Exp \Big[ \big\Vert x^{\mu}(r) - x^{\mu}(r') \big\Vert_{\infty,t}^2 \Big] \to 0$ as $\Vert r'- r \Vert_{\R^d} \searrow 0$, by using the continuity of the initial condition (see \cite[hypothesis (3) p. 7]{cabana-touboul:15}). We then conclude exactly as in \cite[Theorem 9]{cabana-touboul:15} that $r \to \Exp \Big[ d_t\Big((x^{\mu}(r),r),(x^{\nu}(r) ,r)\Big)^2 \Big]$ is continuous, and that $\mu \to Q_{\mu}$ admits a unique fix point relying on equation \eqref{ineq:uniquenessth3} and a Picard's iteration.
\end{proof}

\begin{lemma}\label{lem:existence_uniqueness}
For any $r\in D$ and $\mu\in \M_{1}^{+}(\C\times D)$, there exists a unique strong solution to the SDE:
\begin{equation*}
\left\{
  \begin{array}{ll}
  dx_t^{{\mu}}(r) = f(r,t,x_t^{{\mu}}(r)) dt + \lambda O^{\tilde{W}}_{\mu}(t,x_t^{{\mu}}(r)) dt + \lambda d\tilde{W}_t\\
  x_0^{{\mu}}(r) = \bar{x}_0(r).
  \end{array}
\right.
\end{equation*}
where $\tilde{W}$ is a $\mathbbm{P}$-Brownian motion, $\bar{x}_0(r) \in \R$ is the realization of the continuous version for the family of initial laws $\big(\mu_0(r)\big)_{r \in D}$, and
\[
O^{\tilde{W}}_{\mu}(t,x):= \E_{\gamma} \bigg[ \Lambda_t\big( G^{\mu}(x)\big) G^{\mu}_t(x) \int_0^t G^{\mu}_s(x)\big(d\tilde{W}_s- m_{\mu}(s,x)ds\big) \bigg] + m_{\mu}(t,x).
\]
\end{lemma}

\begin{proof}
The proof relies on Picard's iterations. Let $x^0 \in \C$ with $x^0_0=\bar{x}_0(r)$, and define recursively the sequence $\big(x^n_t, 0\leq t \leq T \big)_{n \in \N^*}$ by
\[
x^{n+1}_t = \bar{x}_0(r) + \int_0^t f(r,s,x_s^{n}) ds + \int_0^t \lambda O^{\tilde{W}}_{\mu}(s,x_s^{n}) ds + \lambda \tilde{W}_t, \; \; \forall t \in [0,T].
\]
As a consequence, for any $t \in [0,T]$ we obtain as in \eqref{ineq:FixPointAverage}
\begin{align*}
x^{n+1}_t -x^n_t & =  \int_0^t \big(f(r,s,x_s^{n})- f(r,s,x_s^{n-1})\big) ds + \int_0^t \lambda \big(O^{\tilde{W}}_{\mu}(s,x_{\cdot}^{n})-O^{\tilde{W}}_{\mu}(s,x_{\cdot}^{n-1})\big) ds\\
& = \int_0^t \big(f(r,s,x_s^{n})+\lambda m_{\mu}(s,x^n_s)- f(r,s,x_s^{n-1}) -\lambda m_{\mu}(s,x^{n-1}_s)\big) ds\\
& + \lambda \int_0^t \int_0^s \Big( \widetilde{K}^s_{\mu,x^n_{\cdot}}(s,v)m_{\mu}(v,x_{\cdot}^n) \widetilde{K}^s_{\mu,x^{n-1}_{\cdot}}(s,v)m_{\mu}(v,x_{\cdot}^{n-1})\Big)dv ds\\
+ \lambda \int_0^t & \bigg\{\E_{\gamma} \bigg[ \Lambda_s\big( G^{\mu}(x^{n}_{\cdot}) \big) G^{\mu}_s(x^{n}_{\cdot}) \Big(\int_0^s G^{\mu}_v(x^n_{\cdot}) d\tilde{W}_v \Big)\bigg] \\
& - \E_{\gamma} \bigg[ \Lambda_s\big( G^{\mu}(x^{n-1}_{\cdot}) \big) G^{\mu}_s(x^{n-1}_{\cdot}) \Big(\int_0^s G^{\mu}_v(x^{n-1}_{\cdot}) d\tilde{W}_v \Big) \bigg]\bigg\} ds.
\end{align*}
Then, using inequalities~\eqref{ineq:uniquenessth1} ans \eqref{ineq:uniquenessth2} to cope with the two first terms of the right-hand side, and controlling the last term as in the proof of theorem~\ref{thm:UniquenessOfTheMinimum}, we find taking the expectation
\[
\Exp \Big[ \Vert \big| x^{n+1}- x^n \Vert_{\infty,t}^2 \Big] \leq C_T \int_0^t \Exp \Big[ \Vert x^{n}- x^{n-1} \Vert_{\infty,s}^2 \Big] ds.
\]
The conclusion now relies on classical arguments.
\end{proof}

\subsection{Convergence of the process and Quenched results.}

We are now in a position to prove theorem~\ref{thm:Convergence}.

\begin{proof}[Proof of Theorem~\eqref{thm:Convergence}]
	Let $\delta>0$ and $B(Q,\delta)$ the open ball of radius $\delta$ centered in $Q$ for the Vaserstein distance. We prove that $Q^{N}(\hat{\mu}_N\notin B(Q,\delta))$ tends to zero exponentially fast as $N$ goes to infinity. In fact, the upper-bound of the LDP for the closed set $B(Q,\delta)^c$ yields
	\[
	\limsup_{N\to\infty} \frac{1}{N} \log Q^{N}(\hat{\mu}_N \notin B(Q,\delta)) \leq -\inf_{B(Q,\delta)^c} H < 0
	\]
	where the last inequality comes from the fact that $H$ attains its unique minimum at $Q$. This implies that $Q^{N}(\hat{\mu}_N \notin B(Q,\delta)) \to 0$ at least exponentially fast, so that the result is proved.
\end{proof}

\begin{proof}[Proof of Theorem~\ref{thm:Quenched}]
For a given closed set $F$, we can obtain a quenched upper-bound as a consequence of Theorem.3 and Borel-Cantelli, by proceeding exactly as in \cite[Theorem 2.7 of Appendix C.]{ben-arous-guionnet:95}. 
As $\M_1^+(\C \times D)$ is Polish, we are able to define a sequence of closed sets $(F_i)_{i \in \N}$ of $\M_1^+(\C \times D)$ such that for all closed set $F\subset \M_1^+(\C \times D)$ there exists $A_F \subset \N$, and
\[
F= \bigcap_{i \in A_F} F_i.
\] 
Moreover, as $\big\{ F \subset \M_1^+(\C \times D), \exists A_F \subset \N, F= \bigcap_{i \in A_F} F_i \big\}$ is countable and contains every closed set, we obtain an $\mathcal{P}$-almost sure upper-bound for every closed set:
\[
\mathcal{P}-a.s, \; \; \forall \text{ closed set } F \subset \M_1^+(\C \times D), \; \;\limsup_{N\to\infty} \frac{1}{N} \log Q^{N}_{\mathbf{r}}(J)(\hat{\mu}_N \in F) \leq -\inf_{F} H.
\] 
$H$ being a good rate function, the $\mathcal{P}$-almost sure exponential tightness is a consequence of \cite[Exercice 4.1.10 (c)]{dembo-zeitouni:09} (citing the results of~\cite[Lemma 2.6]{lynch1987large} and~\cite[Theorem P]{pukhalskii}), whereas the $\mathcal{P}$-almost sure convergence of the empirical measure stems from Borel-Cantelli lemma, noting that for any $\varepsilon>0$, 
\[Q^{N}_{\mathbf{r}}(J)(\hat{\mu}_N\notin B(Q,\varepsilon)) = Q^{N}_{\mathbf{r}}(J)(d_{T}(\hat{\mu}_N,Q)\geq \varepsilon)\]
is summable. 

\end{proof}

\section{Perspectives and Open problems}

In this paper, we have investigated the dynamics of randomly interacting diffusions with complex interactions depending on the state of both particles. From the mathematical viewpoint, we have extended fine estimates on large deviations initially developed for spin-glass systems~\cite{ben-arous-guionnet:95,guionnet:97} to the present setting. The proof proceeds by using a combination of Sanov's theorem and to extend Varadhan's lemma to a functional that does not directly satisfies the canonical assumptions. The limit of the system is a complex non-Markovian process whose dynamics is relatively hard to understand at this level of generality. However, the limits obtained are valid only in the presence of noise, since Girsanov's theorem is used to relate the dynamics of the coupled system to the uncoupled system. The limit of randomly connected systems in the absence of noise is a complex issue with numerous applications, and very little work have been done on this topic. One outstanding contribution that addresses a similar question is the work of Ben Arous, Dembo and Guionnet for spherical spin glass~\cite{ben-arous-dembo-guionnet:01}. In that work, the authors characterize the thermodynamic limit of this system and analyze its long term behavior, providing a mathematical approach for aging. This approach uses the rotational symmetry of the Hamiltonian allowing, by a change of orthogonal basis, to rely on results on the eigenvalues of the coupling matrix. A similar approach seems unlikely to readily extend to the setting of the present manuscript.

In the context of neuroscience, it may be useful to consider spatially extended systems, with delays in the communication, and possibly non-Gaussian interactions. It shall not be hard to combine the methods of the present article to those in~\cite{cabana-touboul:15} and the specific methods developed here to extend the present results to spatially-dependent interactions with space dependent delays. Moreover, we expect that the limit obtained is universal with respect to the distribution of the connectivity coefficient as soon as their tails have a sufficiently fast decay, as demonstrated for a discrete-time neuronal network in~\cite{moynot2002large}. Eventually, the results shall hold in cases where the intrinsic dynamics is not Lipschitz-continuous as soon as sufficient non-explosion estimates are obtained on the solutions of the uncoupled system, as was the case in~\cite{benarous-guionnet:95,guionnet:97}. We however mention that in this case, the original fixed-point method developed in the present article to prove existence and uniqueness of solutions to the mean-field equations are no more valid and adequate methods needs to be used as the ones presented in~\cite{benarous-guionnet:95,guionnet:97}.

\end{document}